\documentclass[11pt, a4paper, notitlepage]{article}

\usepackage{amsmath, latexsym, amsthm, amsfonts, amssymb} 
\usepackage{graphicx} 
\usepackage{epsfig}
\usepackage{varioref}

\theoremstyle{plain}
\newtheorem{thm}{Theorem}[section]

\newtheorem{lem}{Lemma}[section]

\newtheorem{cnd}{Condition}[section]

\theoremstyle{definition} 

\newcommand{\infint}{\int_{-\infty}^{\infty}}
\newcommand{\Log}{\operatorname{Log}}

\def\ex{{\rm E\,}}

\def\supp{\sup_{t\in[-h^{-1},h^{-1}]}}
\def\suppp{\sup_{\mathcal{T}}}

\begin{document}
\date{\today}
\title{Nonparametric inference for discretely sampled L\'evy processes}
\author{Shota Gugushvili\\
{\normalsize Department of Mathematics}\\
{\normalsize Vrije Universiteit Amsterdam}\\
{\normalsize De Boelelaan 1081a}\\
{\normalsize 1081 HV Amsterdam}\\
{\normalsize The Netherlands}\\
{\normalsize s.gugushvili@vu.nl}}
\maketitle
\begin{abstract}
Given a sample from a discretely observed L\'evy process
$X=(X_t)_{t\geq 0}$ of the finite jump activity, the
problem of nonparametric estimation of the L\'evy density $\rho$
corresponding to the process $X$ is studied. An estimator of $\rho$ is proposed that is based
on a suitable inversion of the L\'evy-Khintchine formula and a plug-in device. The main
results of the paper deal with upper risk bounds for estimation of $\rho$ over suitable classes of L\'evy triplets. The corresponding lower bounds are also discussed.
\end{abstract}
\newpage

\section{Introduction}

Recent years have witnessed a great revival of interest in L\'evy
processes, which is primarily due to the fact that they have found numerous
applications in various fields. The main
interest has been in mathematical finance, see e.g.\ \cite{tankov}
for a detailed treatment and many references, however L\'evy processes obtained due
attention also in queueing, telecommunications, extreme value
theory, quantum theory and many others. A thorough exposition of the
fundamental properties of L\'evy processes can be found e.g.\ in \cite{bertoin}, \cite{kyprianou} and
\cite{sato}.

It is well-known that L\'evy processes have a close link with
infinitely divisible distributions: if $X=(X_t)_{t\geq 0}$ is a
L\'evy process, then its marginal distributions are all infinitely
divisible and are determined by the distribution of $X_{\Delta},$
where $\Delta>0$ is an arbitrary fixed number. Conversely, given
an infinitely divisible distribution $\mu,$ one can construct a
L\'evy process $X=(X_t)_{t\geq 0},$ such that $P_{X_{\Delta}}=\mu,$ cf.\ Theorem 7.10
in \cite{sato}. Hence the law of the process $X$ can be uniquely
characterised by the characteristic function of $X_{\Delta},$
where $\Delta>0$ is some fixed number. By the L\'evy-Khintchine
formula for infinitely divisible distributions, the characteristic
function $\phi_{X_{\Delta}}$ of $X_{\Delta}$ can be written as
\begin{equation*}
\phi_{X_{\Delta}}(t)=e^{\psi_{\Delta}(t)},
\end{equation*}
where the exponent $\psi_{\Delta},$ called the characteristic or L\'evy exponent, is given by
\begin{equation}
\label{levyexponent} \psi_{\Delta}(t)=\Delta i\gamma_0
t-\Delta\frac{1}{2}\sigma^2t^2+\Delta\int_{\mathbb{R}\setminus\{0\}}(e^{itx}-1-itx
1_{[|x|\leq 1]})\nu(dx),
\end{equation}
see Theorem 8.1 of \cite{sato}. Here
$\gamma_0\in\mathbb{R},$ $\sigma\geq 0,$ and $\nu$ is a measure
concentrated on $\mathbb{R}\setminus\{0\},$ such that
$\int_{\mathbb{R}\setminus\{0\}}(1\wedge x^2)\nu(dx)<\infty.$ This
measure is called the L\'evy measure, while the triple
$(\gamma_0,\sigma^2,\nu)$ is referred to as the characteristic or
L\'evy triplet of $X.$ The parameter $\gamma_0$ is called a drift parameter
and a constant $\sigma^2$ is a diffusion parameter. The representation in \eqref{levyexponent} in terms
of the L\'evy triplet is unique. It then follows that the L\'evy
triplet determines uniquely the law of any L\'evy process.
Therefore, many statistical inference problems for L\'evy processes can be
reduced to inference on the corresponding characteristic
triplets.

Until quite recently most of the existing literature dealt with
parametric inference procedures for L\'evy processes, see e.g.\
\cite{akritas1}--\cite{basawa2}, \cite{bibby}--\cite{haerdle}, \cite{carr}, \cite{jongbloed},
\cite{nolan}, \cite{rydberg} and \cite{zolotarev}. However, a
nonparametric approach is also possible and arises if one does not
impose parametric assumptions on the L\'evy measure, or its
density, in case the latter exists. A nonparametric approach can
give e.g.\ valuable indications about the shape of the L\'evy density.
Furthermore, parametric inference for L\'evy processes is
complicated by the fact that for many L\'evy processes their
marginal densities are often intractable or not available in
closed form. This makes the implementation of such a standard parameter estimation method as the maximum likelihood method difficult. We refer e.g.\ to \cite{ait-sahalia}, \cite{buchmann}--\cite{bugr}, \cite{chen}, \cite{comtecatalot}--\cite{genon}, \cite{cont}, \cite{gug}, \cite{gug3}, \cite{meulen}--\cite{kappus},  \cite{reiss}, \cite{watteel}, as well as the proceedings \cite{proceedings} and references therein for a nonparametric approach to
inference for L\'evy processes.

In the present work we will assume that the L\'evy measure $\nu$
has a finite total mass, i.e.\ $\nu(\mathbb{R})<\infty,$ and that it has
a density $\rho.$ In essence this means that the L\'evy process
that we sample from is a sum of a linear drift, a
rescaled Brownian motion and a compound Poisson process. Thus this
model is related to Merton's model of an asset price, see
\cite{merton}. Nonparametric inference for a similar model was considered in \cite{belomestny}, \cite{chen} and \cite{gug3}.

Since in our case $\nu(\mathbb{R})<\infty,$ the L\'evy-Khintchine
exponent can be rewritten as
\begin{equation}
\label{levyexponent2} \psi_{\Delta}(t)=\Delta i\gamma
t-\Delta\frac{1}{2}\sigma^2t^2+\Delta\int_{-\infty}^{\infty}(e^{itx}-1)\rho(x)dx.
\end{equation}
The triple $(\gamma,\sigma^2,\rho)$ is again referred to as a L\'evy triplet. Note that $\gamma$ in \eqref{levyexponent2} differs from $\gamma_0$ in \eqref{levyexponent}.

Suppose that the L\'evy process $X=(X_t)_{t\geq 0}$ is observed at
discrete time instances $\Delta,2\Delta,\ldots,n\Delta,$ with
$\Delta$ kept fixed.  This sampling case is usually referred to as the low frequency data case. For the case when $\Delta$ is allowed to depend on $n$ and $\Delta\rightarrow 0,n\Delta\rightarrow\infty$ as $n\rightarrow\infty$ see e.g.\ \cite{catalot}, \cite{genon} or \cite{figueroalop}. In this case it is customary to talk about high frequency data case. Returning to the case with a fixed $\Delta,$ by a rescaling argument, without loss of
generality, we can take $\Delta=1.$ Based on observations $X_1,\ldots,X_n,$ our goal in this paper is to estimate nonparametrically the L\'evy
density $\rho.$ Notice that this is an inverse problem in that $\rho$ is associated with jump sizes  of a L\'evy process and their intensity, the jumps themselves are not directly observable under the present sampling scheme, and consequently $\rho$ has to be estimated from indirect observations $X_1,\ldots,X_n.$

We will base our estimator of $\rho$ on a suitable
inversion of $\phi_{X_1}.$ The idea of expressing the L\'evy
measure or the L\'evy density in terms of $\phi_{X_1}$ and then
replacing $\phi_{X_1}$ by its natural nonparametric estimator, the
empirical characteristic function, to obtain a plug-in type
estimator for the L\'evy measure or the L\'evy density has been
successfully applied e.g.\ in \cite{chen}, \cite{comte},
\cite{gug}, \cite{gug3}, \cite{reiss} and \cite{watteel}. The
logic behind this approach is that except of some particular
cases, e.g.\ that of the compound Poisson process, see \cite{bu}
and \cite{bugr}, finding an explicit relationship expressing the L\'evy
measure or its density directly in terms of the distribution of
$X_1$ without referring to the Fourier transforms is difficult. This hampers the use of a plug-in device,
which is one of the most popular and useful methods for obtaining
estimators in statistics. On the other hand the Fourier approach
allows one to cover a large class of examples, as shown in the
above-mentioned papers.

Observe that the model we consider in the present work
shares many features characteristic of a convolution model
with partially or totally unknown error distribution, see \cite{matias}, \cite{lacour}, \cite{meister} and  \cite{neumann}. For instance, the Gaussian components in $X_1,\ldots,X_n$ in our case will play a role similar to the measurement error in those papers, in case the latter has a normal distribution.

We proceed to the construction of an estimator of $\rho.$ First by differentiating the L\'evy-Khintchine formula we will
derive a suitable inversion formula for $\rho.$ Suppose that
$\int_{\mathbb{R}}x^2\rho(x)dx<\infty.$ Since $\rho$ has a finite
second moment, so does $X_{1}$ by Corollary 25.8 in \cite{sato}.
Also $\ex[|X_1|]$ is finite by the Cauchy-Schwarz inequality. Hence we can differentiate
$\phi_{X_1}$ with respect to $t$ to obtain
\begin{equation}
\label{psiderivative}
\phi_{X_1}^{\prime}(t)=\phi_{X_1}(t)\left(i\gamma-\sigma^2t+i\int_{-\infty}^{\infty}e^{itx}x\rho(x)dx\right).
\end{equation}
Notice that differentiation of $\int_{\mathbb{R}}(e^{itx}-1)\rho(x)dx$ under the integral sign is justified by the dominated convergence theorem, applicable because of our assumptions on $\rho.$ Next rewrite \eqref{psiderivative} as
\begin{equation}
\label{ratio}
\frac{\phi_{X_1}^{\prime}(t)}{\phi_{X_1}(t)}=i\gamma-\sigma^2t+i\int_{\mathbb{R}}e^{itx}x\rho(x)dx,
\end{equation}
which is possible, because $\phi_{X_1}(t)\neq 0$ for all $t\in\mathbb{R},$ see e.g.\ Theorem $7.6.1$ in \cite{chung}. Differentiating both sides of this identity with respect to $t,$ we get
\begin{equation}
\label{psiderivative2}
\frac{\phi_{X_1}^{\prime\prime}(t)\phi_{X_1}(t)-(\phi_{X_1}^{\prime}(t))^2}{(\phi_{X_1}(t))^2}=-\sigma^2-\int_{\infty}^{\infty}e^{itx}x^2\rho(x)dx,
\end{equation}
where again we interchanged the differentiation and integration order in the righthand side of \eqref{ratio} to obtain the righthand side of \eqref{psiderivative2}. Thus by rearranging the terms we have
\begin{equation}
\label{ast*}
\int_{-\infty}^{\infty} e^{itx}x^2\rho(x)dx=\frac{(\phi_{X_1}^{\prime}(t))^2-\phi_{X_1}^{\prime\prime}(t)\phi_{X_1}(t)}{(\phi_{X_1}(t))^2}-\sigma^2.
\end{equation}
Suppose that the righthand side is integrable, which is implied by the assumption that $\phi_{\rho}^{\prime\prime}$ is integrable. Here $\phi_{\rho}$ denotes the Fourier transform of $\rho.$ Then by the Fourier inversion argument the relationship
\begin{equation*}
x^2\rho(x)=\frac{1}{2\pi}\int_{-\infty}^{\infty} e^{-itx}\left(\frac{(\phi_{X_1}^{\prime}(t))^2-\phi_{X_1}^{\prime\prime}(t)\phi_{X_1}(t)}{(\phi_{X_1}(t))^2}-\sigma^2\right)dt
\end{equation*}
holds. If $x\neq 0,$ this yields
\begin{equation}
\label{inv}
\rho(x)=\frac{1}{2\pi x^2}\infint
e^{-itx}\left(\frac{(\phi_{X_1}^{\prime}(t))^2-\phi_{X_1}^{\prime\prime}(t)\phi_{X_1}(t)}{(\phi_{X_1}(t))^2}-\sigma^2\right)dt,
\end{equation}
and we obtain a desired inversion formula. This formula coincides with the one given in \cite{burnaev}\footnote{\cite{burnaev} contains a more general result valid also for L\'evy densities with infinite total mass. However, the statement of the theorem in
\cite{burnaev} mistakenly claims that the L\'evy density $\rho$ is
bounded under the assumptions given in \cite{burnaev}. In reality this can in general be ascertained only for $x^2\rho(x).$ Examples (e) and (f) considered in \cite{burnaev} illustrate our point.}. The formula has to be
compared to related inversion formulae given in
\cite{comte}, \cite{genon}, \cite{reiss} and \cite{watteel}. Notice that under stronger moment conditions on $X_1$ one can perform the differentiation step in the above derivation not twice, but three times, thereby eliminating $\sigma^2$ from \eqref{psiderivative2}, and one can obtain an inversion formula of the same type as in \eqref{inv}, but not involving $\sigma^2$ explicitly, see e.g.\ \cite{genon}. We do not pursue this path, as a study of asymptotic properties of an estimator of $\rho$ of the same type as we propose below based on this different inversion formula would require stronger moment conditions on $X_1,$ cf.\ the discussion in the next section. It would also involve longer and more technical proofs of the asymptotic results. Finally, under certain smoothness assumptions on the L\'evy density it would lead to an estimator with worse convergence rate than the one that we propose below. See Section \ref{results} for an additional discussion.

Denote $Z_j=X_{j}-X_{j-1}$ and observe that $Z_1,\ldots,Z_n$ are i.i.d., which
follows from the stationary independent increments property of a
L\'evy process. Let
$\hat{\phi}(t)=n^{-1}\sum_{j=1}^n e^{i tZ_j}.$
By the strong law of large numbers, for every fixed
$t,$ the empirical characteristic function
$\hat{\phi}(t)$ and its derivatives with respect to $t,$
$\hat{\phi}^{\prime}(t)$ and
$\hat{\phi}^{\prime\prime}(t),$ converge a.s.\ to
$\phi_{X_1}(t),\phi_{X_1}^{\prime}(t)$ and
$\phi_{X_1}^{\prime\prime}(t),$ respectively. Using a
plug-in device, a possible estimator of $\rho(x)$ could then be
\begin{equation}
\label{naivefn} \frac{1}{2\pi
x^2}\infint e^{-itx}\left(\frac{(\hat{\phi}^{\prime}(t))^2}{(\hat{\phi}(t))^2}-\frac{\hat{\phi}^{\prime\prime}(t)}{\hat{\phi}(t)}-\hat{\sigma}^2\right)dt,
\end{equation}
where $\hat{\sigma}^2$ is some estimator of ${\sigma}^2.$ The problem with this `estimator' of $\rho$ is that in general the integrand
in \eqref{naivefn} is not integrable. Furthermore, small values of
$\hat{\phi}(t)$ might render the estimator
numerically unstable, since $\hat{\phi}(t)$ appears in
the denominator in \eqref{naivefn}. Therefore, as an estimator of
$\rho$ we propose the following modification of \eqref{naivefn},
\begin{equation}
\label{fn} \hat{\rho}(x)=\frac{1}{2\pi
x^2}\infint e^{-itx}\left(\frac{(\hat{\phi}^{\prime}(t))^2}{(\hat{\phi}(t))^2}1_{G_t}-\frac{\hat{\phi}^{\prime\prime}(t)}{\hat{\phi}(t)}1_{G_t}-\hat{\sigma}^2\right)\phi_w(ht)dt.
\end{equation}
Here $\phi_w$ denotes the Fourier transform of a kernel function
$w,$ while a number $h>0$ denotes a bandwidth. This terminology is borrowed from the kernel estimation theory, see e.g.\ \cite{tsyb}. The integral in
\eqref{fn} is finite under the assumption that $\phi_w$ has
a compact support, for instance on $[-1,1].$ We define the set $G_t$ in \eqref{fn} by
\begin{equation}
\label{Gt}
G_t=\left\{|\hat{\phi}(t)|\geq \kappa_n
e^{-\Sigma^2/(2h^2)}\right\}.
\end{equation}
Hence $G_t$ depends on $h,$ as well as a constant $\Sigma$ and a
sequence $\kappa_n\rightarrow 0$ of real numbers to be specified in the next section, where we also give some additional heuristics for the definition of $G_t.$ A general reason for using truncation with $1_{G_t}$ is a desire of numerical stability, but truncation in \eqref{fn} will also help in proving the asymptotic results from Section \ref{results}. At this point notice that we could have also used a ``diagonal-out" estimator
\begin{equation*}
\frac{2}{n(n-1)}\sum_{1\leq j<k\leq n}e^{itZ_j}e^{itZ_k}
\end{equation*}
to estimate $(\phi_{X_1}(t))^2$ in the denominator of \eqref{inv} and a similar ``diagonal-out" estimator to estimate $(\phi_{X_1}^{\prime}(t))^2.$ An advantage of these two estimators is that they are unbiased estimators of $(\phi_{X_1}(t))^2$ and $(\phi_{X_1}^{\prime}(t))^2,$ respectively, while $(\hat{\phi}(t))^2$ and $(\hat{\phi}^{\prime}(t))^2$ are not. On the theoretical side study of the resulting modification of $\hat{\rho}$ would require the use of the theory of U-statistics, see e.g.\ Chapter $12$ in \cite{vdvaart}. However, since in the present paper we are mainly concerned with rates of convergence for estimation of $\rho,$ we refrain from studying this possible modification of $\hat{\rho}.$

It
remains to propose an estimator of $\sigma^2.$ To this end we use
an estimator from \cite{gug3} defined via
\begin{equation}
\label{sig2}
\hat{\sigma}^2=\int_{\mathbb{R}}\max\{\min\{M_n,\log
(|\hat{\phi}(t)|)\},-M_n\}v_h(t)dt.
\end{equation}
Here $v_h$ is a kernel function depending on $h,$ while $M_n$ denotes a sequence of
positive numbers diverging to infinity at a suitable rate. Appropriate conditions on all three will be given in the next section. The estimator $\hat{\sigma}^2$ is again based on the L\'evy-Khintchine formula and we refer to \cite{gug3} for the heuristics of its introduction. There does not seem to exist an `easy' way to define an estimator of $\sigma^2.$ `Nonparametric' estimators of finite-dimensional parameters in semiparametric deconvolution problems (these are related to the problem we are considering in the present paper) have already been proposed in the literature, see e.g.\ \cite{matias} and \cite{atomic}. In the context of L\'evy processes `nonparametric' estimators of finite-dimensional parameters have been used e.g.\ in \cite{belomestny} and \cite{gug3}. These estimators can often be proven to be rate-optimal.

If $\phi_w$ is symmetric and real-valued, then by taking a complex conjugate one can see that $\hat{\rho}$ is
real-valued, because this amounts to changing the integration
variable from $t$ into $-t$ in \eqref{fn}. On the other hand,
positivity of $\hat{\rho}$ is not guaranteed, which is a slight drawback often shared by estimators based on Fourier inversion and kernel smoothing. However, one can always
consider $\hat{\rho}^+(x)=\max(\hat{\rho}(x),0)$ instead of $\hat{\rho}(x).$ For this
modified estimator we have $\ex[(\hat{\rho}^+(x)-\rho(x))^2]\leq
\ex[(\hat{\rho}(x)-\rho(x))^2]$ and hence its performance is at
least as good as that of $\hat{\rho},$ if the mean square error is
used as the performance criterion. We restrict our attention to studying the estimator $\hat{\rho}$ only.

The structure of the paper is as follows: in the next section we
will study the asymptotic behaviour of the mean square error of
the proposed estimator of $\rho.$ In particular we will derive convergence rates of our estimator over appropriate classes of L\'evy triplets and discuss the corresponding lower bounds for estimation of $\rho.$ The section is concluded with a discussion on the obtained results and possible extensions. The proofs of results
from Section \ref{results} are collected in Section \ref{proofs}.

\section{Results}
\label{results}

We first formulate conditions that will be used to establish
asymptotic properties of the estimator $\hat{\rho}.$ We also
supply some comments on these conditions. Introduce a jump size
density $f(x):=\rho(x)/\nu(\mathbb{R}).$

\begin{cnd}
\label{conditionf} Let the unknown L\'evy density $\rho$ belong to the class
\begin{align*}
W(\beta,L,L^{\prime},L^{\prime\prime},K,\Lambda)=\Bigl\{ &\rho:\rho(x)=\nu(\mathbb{R})
f(x),f\text{ is a probability density},\\
&\infint |t|^{\beta}|\phi_f(t)|dt\leq L,\\
&|\phi_f(t)|\leq \frac{L^{\prime}}{|t|^{\beta}},\\
&|\phi_f^{\prime}(t)|\leq \frac{L^{\prime\prime}}{|t|^{\beta}},\\
&\infint x^{12}f(x)dx\leq K,\\
&\phi_f^{\prime\prime}\text{ is integrable},\\
&\nu(\mathbb{R})\in(0,\Lambda]\Bigr\},
\end{align*}
where $\beta,L,L^{\prime},L^{\prime\prime},K$ and $\Lambda$ are strictly positive numbers.
\end{cnd}
This condition is similar to the one given in \cite{gug3} and we
refer to the latter for additional discussion. When $\beta$ is an
integer, the integrability condition on $\phi_f$ in Condition \ref{conditionf} is roughly
equivalent to $f$ having a derivative of order $\beta.$ The moment condition on $f,$ and consequently on $\rho,$ is admittedly strong, but on the other hand in mathematical finance it is customary to assume that $\rho$ has a finite exponential moment. The moment condition in Condition \ref{conditionf} is used to prove an appropriate maximal inequality for $\hat{\phi}$ and its derivatives, see Theorem \ref{thm-ineq}, which constitutes one of the important working tools of the paper.

\begin{cnd}
\label{conditionsigma} Let $\sigma$ be such that
$\sigma\in[0,\Sigma],$ where $\Sigma$ is a strictly
positive number.
\end{cnd}
For the case when $\Sigma=0,$ that is to say when $\sigma=0$ is known beforehand, we refer to \cite{comte} and \cite{gug}.  Observe
that in general $\sigma$ determines how fast the characteristic
function $\phi_{X_1}$ decays at plus and minus infinity, because as it is easy to see, one has
\begin{equation}
\label{twost}
|\phi_{X_1}(t)|\geq e^{-2\Lambda-\Sigma^2 t^2/2}.
\end{equation}
The knowledge of
$\Sigma,$ which we will assume, gives us a lower bound on the rate of decay of
$\phi_{X_1}$ at plus and minus infinity (uniformly in $\sigma\in[0,\Sigma]$).

\begin{cnd}
\label{conditiongamma} Let $\gamma$ be such that
$|\gamma|\leq \Gamma,$ where $\Gamma$ is a positive number.
\end{cnd}

This condition is the same as the one in \cite{gug3}, cf.\ also \cite{belomestny}.

\begin{cnd}
\label{conditionh} Let the bandwidth $h=h_n$ depend on $n$ and be
such that $h_n=(\eta\log n)^{-1/2}$ with $0<\eta<1/(2\Sigma^{2}).$
\end{cnd}

This condition is similar to the one given in \cite{gug3}. Notice that in order to keep our notation compact, we will suppress the dependence of $h_n$ on $n$ in the notation. The fact that the bandwidth $h$ depends on $\Sigma$
has a parallel in the condition on the smoothing parameter in
\cite{comte}, see Remark $4.2$ there, and also arises in
deconvolution problems with unknown error distribution, see
\cite{matias}. As usual in kernel estimation, see e.g.\ p.\ 7 in \cite{tsyb}, a choice of $h$ establishes a trade-off between the bias and the variance of the estimator: too small an $h$ will result in an estimator with small bias but large variance, while too large an $h$ results in the estimator with large bias but small variance. From Theorems \ref{thm-rho} and \ref{lowbound-thm} it will follow that the choice of $\rho$ as in Condition \ref{conditionh} is optimal in one particular situation in a sense that it asymptotically minimises the order of the mean square error of the estimator $\hat{\rho}$ at a fixed point $x.$

\begin{cnd}
\label{conditionk} Let the kernel $w$ be the sinc kernel:
$w(x)=\sin x/(\pi x).$
\end{cnd}

The sinc kernel has also been used in \cite{gug3} when estimating the L\'evy density. Its use is frequent in deconvolution problems, see e.g.\ \cite{matias}. The Fourier transform of the sinc kernel is given by
$\phi_w(t)=1_{[-1,1]}(t).$

\begin{cnd}
\label{conditionkappa} Let the sequence $\kappa_n$ be such that
$\kappa_n=\kappa |\log h|^{-1}$ for a constant
$\kappa>0.$
\end{cnd}

This is a technical condition used in the proofs. Other sufficiently slowly vanishing sequences $\{\kappa_n\}$ can also be used, ours is just one concrete example. The intuition
behind Condition \ref{conditionkappa} is that up to a constant $e^{-2\Lambda},$ the term $e^{-\Sigma^2/(2h^2)}$ gives a
lower bound for the absolute value of the characteristic
function $\phi_{X_1}(t)$ on the interval $[-h^{-1},h^{-1}],$ cf.\ \eqref{twost}. For
$n$ large enough, with an indicator $1_{G_t}$ in the definition of
$\hat{\rho}$ we thus cut-off those frequencies $t$ for which
$|\hat{\phi}(t)|$ becomes smaller than the lower bound for
$|\phi_{X_1}(t)|$ over $t\in[-h^{-1},h^{-1}].$ Of course different truncation methods are also possible and we refer e.g. to \cite{comte} for an alternative truncation method in the definition of an estimator of a L\'evy density in a problem similar to ours. We think that it is natural to incorporate the knowledge of $\Sigma$ in the selection of the threshold in \eqref{fn}, since the knowledge of $\Sigma$ is required anyway when selecting the bandwidth $h.$ With our choice of $h$ the set $G_t$ can also be characterised in terms of the sample size $n,$ because $h$ is a function of $n,$ see Condition \ref{conditionh}. Thus our truncation method is not dissimilar from the one in the deconvolution problem studied in \cite{neumann}.

Next we recall two conditions from \cite{gug3}, which were used to
study the asymptotics of the estimator $\hat{\sigma}^2.$ For the convenience of a reader we
also state a result on the asymptotic behaviour of its mean square
error. The latter is used in the proof of Theorem \ref{thm-rho}
below.

\begin{cnd}
\label{conditionv}
Let the kernel $v_h(t)=h^3v(ht),$ where the function $v$ is continuous and real-valued, has a support on $[-1,1]$ and is such that
\begin{equation*}
\int_{-1}^1v(t)dt=0, \quad \int_{-1}^1\left(-\frac{t^2}{2}\right)v(t)dt=1, \quad v(t)=O(t^{\beta}) \text{ as } t\rightarrow 0.
\end{equation*}
Here $\beta$ is the same as in Condition \ref{conditionf}.
\end{cnd}

It is for simplicity of the proofs that we assume that the smoothing parameter $h$ in the definition of $\hat{\sigma}^2$ is the same as in Condition \ref{conditionsigma}. In practice the two need not be equal, although they have to be of the same order.

\begin{cnd}
\label{conditionm} Let the truncating sequence $M=(M_n)_{n\geq 1}$
be such that $M_n=m_n h^{-2},$ where $m_n=|\log h|^{-1}.$
\end{cnd}

Here we implicitly assume that $n$ is large enough, so that $m_n$ is real and $m_n>0.$ Other conditions are also possible, ours is just one concrete example. The use of the truncation in the definition of $\hat{\sigma}^2$ in \eqref{sig2} is that it prevents the estimator from exploding: $|\hat{\phi}(t)|$ can in general take arbitrarily small values and $\log(|\hat{\phi}(t)|)$ consequently can become arbitrarily large.

In the remainder of the paper we will often use the symbols $\lesssim$ and $\gtrsim$ when comparing two sequences $a_n$ and $b_n,$ respectively meaning $a_n$ is less or equal than $b_n,$ or $a_n$ is greater or equal than $b_n$ up to a constant that does not depend on $n.$ The symbol $\asymp$ will be used to denote the fact that two sequences of real numbers are asymptotically of the same order.

\begin{thm}
\label{thm-sigmatilde} Denote by $\mathcal{T}$ the collection of
all L\'evy triplets satisfying Conditions
\ref{conditionf}--\ref{conditiongamma} and assume Conditions
\ref{conditionh}, \ref{conditionv} and \ref{conditionm}. Let the estimator
$\hat{\sigma}^2$ be defined by \eqref{sig2}. Then
\begin{equation*}
\suppp \ex[(\hat{\sigma}^2-\sigma^2)^2]\lesssim(\log
n)^{-\beta-3}
\end{equation*}
holds.
\end{thm}

Even though Condition \ref{conditionf} differs slightly from its counterpart in \cite{gug3}, this does not affect the proof of Theorem \ref{thm-sigmatilde}.   Although the convergence rate of the estimator $\hat{\sigma}^2$ is logarithmic, the contribution of $\hat{\sigma}^2$ to an upper bound on the mean square error of $\hat{\rho}(x)$ is asymptotically negligible compared to other terms, as can be seen from the proof of Theorem \ref{thm-rho}. By techniques similar to those used in \cite{atomic} in a related deconvolution problem, it is expected that under the same conditions on the class of L\'evy triplets as in Theorem \ref{thm-sigmatilde} one can prove that $\hat{\sigma}^2$ is rate-optimal, but since our emphasis in the present work is on estimation of a L\'evy density, we refrain from studying this question. For additional discussion on the estimator $\hat{\sigma}^2$ see \cite{gug3}.

Notice that had we not assumed $\nu(\mathbb{R})\leq \Lambda<\infty,$ there would not exist a uniformly consistent estimator of $\sigma^2,$ see Remark $3.2$ in \cite{reiss}. In fact even the existence of a consistent estimator of $\sigma^2$ is not clear in that general setting.

Together with the above theorem, an important tool in studying the estimator $\hat{\rho}$ is the following maximal inequality for the empirical characteristic function $\hat{\phi}(t)$ and its derivatives. Set $\hat{\phi}^{(0)}(t)=\hat{\phi}(t)$ and likewise $\phi_{X_1}^{(0)}(t)=\phi_{X_1}(t).$

\begin{thm}
\label{thm-ineq} Let $k\geq 0$ and $r\geq 1$ be integers. Then we have
\begin{multline}
\label{ineq-1}
\ex\left[
\left(\sup_{t\in[-h^{-1},h^{-1}]}|\hat{\phi}^{(k)}(t)-{\phi}^{(k)}_{X_1}(t)|\right)^r
\right]\\ \lesssim (\||x|^{k+1}\|_{\mathbb{L}_{2\vee
r}(\operatorname{P})}^r + \||x|^{k}\|_{\mathbb{L}_{2\vee
r}(\operatorname{P})}^r) \frac{1}{h^r n^{r/2}},
\end{multline}
provided $\||x|^{k+1}\|_{\mathbb{L}_{2\vee
r}(\operatorname{P})}$ is finite. Here the probability $\operatorname{P}$ on the righthand side refers to the law of $X_1,$ which is uniquely characterised by the triplet $(\gamma,\sigma^2,\rho).$
\end{thm}

The theorem constitutes a generalisation of the corresponding
result for $\hat{\phi}$ and $r=2$ given in \cite{gug3}. The theorem is of possible general interest as well. For related results on the empirical characteristic function see Theorem 1 in \cite{devroye} and Theorem 4.1 in \cite{reiss}.

Equipped with the above two theorems, we are now ready to formulate the first main result of the
paper, which concerns the mean square error of the estimator
$\hat{\rho}$ at a fixed point $x\neq 0.$ Notice that we prefer to work with asymptotics uniform in L\'evy triplets, since existence of the superefficiency phenomenon in nonparametric estimation makes it difficult to interpret fixed parameter asymptotics, see e.g.\ \cite{brown} for a discussion. This also explains why we imposed certain smoothness assumptions on the class of L\'evy densities: too large a class of densities, e.g.\ of all continuous densities, usually cannot be handled when dealing with uniform asymptotics, see e.g. Theorem $1$ on p.\ $36$ in \cite{gyorfi} for an example from probability density estimation.

\begin{thm}
\label{thm-rho} Denote by $\mathcal{T}$ the collection of all
L\'evy triplets satisfying Conditions
\ref{conditionf}--\ref{conditiongamma} and assume Conditions
\ref{conditionh}--\ref{conditionm}. Let the estimator $\hat{\rho}$
be defined by \eqref{fn}. Then we have
\begin{equation*}
\suppp \operatorname{E}[(\hat{\rho}(x)-\rho(x))^2]\lesssim (\log
n)^{-\beta}
\end{equation*}
for every fixed $x\neq 0.$
\end{thm}

Thus the convergence rate of our estimator turns out to be
logarithmic, just as for the estimator of $\rho$ proposed in \cite{gug3}.
This result can be easily understood on an intuitive level by
comparison to a nonparametric deconvolution problem: if the
distribution of the measurement error in a deconvolution model is
normal, and if the class of the target densities is massive enough,
e.g.\ some H\"older or Sobolev class (see Definitions 1.2 and 1.11 in \cite{tsyb}), the
minimax convergence rate for estimation of an unknown density will
be logarithmic for both the mean squared error and mean integrated
squared error as measures of risk, see \cite{fan1} and
\cite{fan2}. Of course the same holds true also for deconvolution
models with unknown error variance, see \cite{matias} and
\cite{meister}. Exactly as kernel-type estimators in semiparametric deconvolution problems, our estimator $\hat{\rho}$ also involves division by an estimator of a characteristic function (or to be more precise by its square), a slight difference being that in semiparametric deconvolution problems we divide by an estimator of the characteristic function of the measurement error variable, while in the definition of $\hat{\rho}$ we divide by $\hat{\phi},$ an estimator of $\phi_{X_1}.$ For large enough $n$ the empirical characteristic function $\hat{\phi}$ should be close to the true characteristic function $\phi_{X_1}$ on the interval $[-h^{-1},h^{-1}].$ Since up to a constant term, $\phi_{X_1}$ behaves at plus and minus infinity as a normal characteristic function, the logarithmic convergence rate of the estimator $\rho$ is then no surprise. Exactly as in normal deconvolution problem over a H\"older or Sobolev class of densities, cf.\ \cite{fan1} and \cite{fan2}, it is due to the dominating squared bias of $\hat{\rho},$ i.e.\ roughly speaking the term $T_1$ in the proof of Theorem \ref{thm-rho}. More formally, in the theorem given below we actually prove that our estimator $\hat{\rho}$ attains the minimax convergence rate for estimation of the L\'evy density $\rho$ at a fixed point $x$ over a suitable class of L\'evy triplets when the risk is measured by the mean square error.

\begin{thm}
\label{lowbound-thm} Let $T$ be a L\'evy triplet
$(\gamma,\sigma^2,\rho),$ such that
$|\gamma|\leq\Gamma,$ $\sigma\in[0,\Sigma],$ $\nu(\mathbb{R})\in(0,\Lambda],$ where $\Gamma,\Sigma$ and $\Lambda$ are strictly positive constants.
Assume furthermore that
\begin{equation}
\label{fcnd} \infint |t|^{\beta}|\phi_f(t)|dt\leq L; \quad |\phi_f(t)|\leq\frac{L^{\prime}}{|t|^{\beta}}; \quad |\phi_f^{\prime}(t)|\leq\frac{L^{\prime}}{|t|^{\beta}}
\end{equation}
for strictly positive constants $\beta,L,L^{\prime}$ and $L^{\prime\prime}.$ Let $\mathcal{T}$ be a collection of all L\'evy
triplets satisfying these conditions. Then for every fixed $x\neq 0$ we have
\begin{equation}
\label{lowerbound}
\inf_{\widetilde{\rho}_n}\sup_{\mathcal{T}}\operatorname{E}[(\widetilde{\rho}_{n}(x)-\rho(x))^2]\gtrsim
(\log n)^{-\beta},
\end{equation}
where the infimum on the lefthand side is taken over all estimators
$\widetilde{\rho}_n$ based on observations $X_1,\ldots,X_n.$
\end{thm}

The proof of the theorem is such that it also works for the case when $\sigma>0$ is assumed known and is fixed. Therefore the knowledge of $\sigma$ does not lead to some estimator of $\rho$ with a better rate of convergence. This is unlike the semiparametric deconvolution problem with unknown error variance, see \cite{matias}, where the fact that the measurement error variance is unknown slows down even further the convergence rate. Disregarding the moment condition in Condition \ref{conditionf}, an easy consequence of Theorems \ref{thm-rho} and \ref{lowbound-thm} is that $\hat{\rho}$ is rate-optimal.

A slow, logarithmic convergence rate of $\hat{\rho}$ seems to indicate that samples of very large size are needed to accurately estimate $\rho.$ However, it is known that in deconvolution problems kernel-type density estimators perform well for reasonable sample sizes, provided the noise term variance is not too large, see e.g.\ \cite{delaigle0}, \cite{shota} or \cite{wand}. Likewise, a spectral cut-off method of \cite{belomestny} and \cite{belom} produces good results for small values of $\sigma$ in the problem of calibration of exponential L\'evy models. Since in the financial setting it is perhaps unnatural to assume that $\sigma$ is known and $\sigma\rightarrow 0$ as $n\rightarrow\infty,$ which constitutes the mathematical formalisation of the statement that in the asymptotic setting the noise level is low, and since in the present work we are mainly concerned with asymptotics, we will explore a different possibility, namely that the L\'evy density is much smoother than the H\"older or Sobolev class L\'evy densities. Our results will parallel those from \cite{but}, where it is shown in the deconvolution context that better than logarithmic convergence rates can be obtained in case when the target density is supersmooth itself, i.e.\ essentially has a characteristic function that decays exponentially fast at plus and minus infinity.

We first give a condition on the class of L\'evy densities.

\begin{cnd}
\label{conditionf2} Let the unknown L\'evy density $\rho$ belong to the class
\begin{align*}
A(\alpha,s,L,L^{\prime},L^{\prime\prime},K,\Lambda)=\Bigl\{ &\rho:\rho(x)=\nu(\mathbb{R})
f(x),f\text{ is a probability density},\\
&\infint |\phi_f(t)|^2 \exp(2\alpha |t|^s) dt\leq L,\\
&\frac{|\phi_f(t)|}{|t|^{(1-s)/2}e^{-\alpha |t|^s}}\leq L^{\prime} ,\\
&\frac{|\phi_f^{\prime}(t)|}{|t|^{(1-s)/2}e^{-\alpha |t|^s}}\leq L^{\prime\prime} ,\\
&\infint x^{12}f(x)dx\leq K,\\
&\phi_f^{\prime\prime}\text{ is integrable},\\
&\nu(\mathbb{R})\in(0,\Lambda]\Bigr\},
\end{align*}
where $\alpha,s,L,K$ and $\Lambda$ are strictly positive numbers.
\end{cnd}

The `size' of the class $A(\alpha,s,L,L^{\prime},L^{\prime\prime},K,\Lambda)$ is much smaller than the `size' of the class $W(\beta,L,L^{\prime},L^{\prime\prime},K,\Lambda),$ and it is intuitively clear that better convergence rates can be expected for estimation of $\rho$ over the former class than over the latter class. We will refer to the class $A(\alpha,s,L,L^{\prime},L^{\prime\prime},K,\Lambda)$ as the class of supersmooth L\'evy densities.

Since the estimator $\hat{\rho}$ depends on the estimator $\hat{\sigma}^2,$ we first need to study the asymptotics of the latter. With a different class of L\'evy densities than in Theorem \ref{thm-sigmatilde}, the conditions on the bandwidth $h$ and kernel $v_h$ have to be modified accordingly. These are supplied below.

\begin{cnd}
\label{conditionh2} Let the bandwidth $h$ depend on $n$ and be
such that $h$ is a positive solution of the equation
\begin{equation}
\label{bandwidth2}
\frac{2\alpha}{h^s}+\frac{2\Sigma^2}{h^2}=\log n - (\log\log n)^2.
\end{equation}
\end{cnd}
Here we thus suppose that $s$ is known. We also assume that $n$ is large enough, so that equation \eqref{bandwidth2} indeed has a positive root. Condition \ref{conditionh2} is motivated by a similar condition on the bandwidth in the deconvolution problem studied in \cite{but}. An optimal bandwidth, i.e. a bandwidth that asymptotically minimises the risk of the estimator (or an upper bound on it), is typically computed in kernel estimation by differentiating an upper bound on the risk of the estimator with respect to $h,$ setting the derivative to zero and solving $h$ from the obtained equation. However, in our case an optimal $h$ can also be computed from \eqref{bandwidth2}, cf.\ Section 3 in \cite{but}, and we give the corresponding argument in the proof of Theorem \ref{thm-rho2}. The two methods of course yield the same asymptotic results.

\begin{cnd}
\label{conditionv2}
Let the kernel $v_h(t)=h^3v(ht),$ where the function $v$ is continuous and real-valued, has a support on $[-\sqrt{2},-1]\bigcup[1,\sqrt{2}]$ and is such that
\begin{equation*}
\int_{\mathbb{R}}v(t)dt=0, \quad \int_{\mathbb{R}}\left(-\frac{t^2}{2}\right)v(t)dt=1.
\end{equation*}
\end{cnd}

Instead of defining the support of $v$ by $[-\sqrt{2},-1]\bigcup[1,\sqrt{2}],$ we could have defined it as $[-a,-1]\bigcup[1,a]$ for $1<a\leq\sqrt{2},$ which would result in a better convergence rate for $\hat{\sigma}^2.$ However, $a=\sqrt{2}$ actually suffices for the purpose of estimation of $\rho,$ as a contribution of $\hat{\sigma}^2$ to an upper bound on the risk of $\hat{\rho}$ will still be asymptotically of at most the same order as that of other terms, cf.\ the proof of Theorem \ref{thm-rho2}. We do not address the problem of constructing a rate-optimal estimator of $\sigma^2$ in the present paper.

The following result holds true.

\begin{thm}
\label{thmsigmatilde2}
Denote by $\mathcal{T}$ the collection of
all L\'evy triplets satisfying Conditions
\ref{conditionsigma}, \ref{conditiongamma} and \ref{conditionf2} and assume that Conditions \ref{conditionm}, \ref{conditionh2} and  \ref{conditionv2} hold. Let $s<2$ and let the estimator
$\hat{\sigma}^2$ be defined by \eqref{sig2}. Then
\begin{equation*}
\suppp \ex[(\hat{\sigma}^2-\sigma^2)^2]\lesssim h^{s+5}\exp\left( -\frac{2\alpha}{h^s}\right)
\end{equation*}
holds, where $h$ is defined in Condition \ref{conditionh2}.
\end{thm}

The asymptotics of the estimator $\hat{\sigma}^2$ (and also those of $\hat{\rho}$) change qualitatively when $s>2.$ In particular, the convergence rate of $\hat{\sigma}^2$ becomes polynomial. Although supersmooth densities with $s>2$ are in principle conceivable, they do not include well-known representatives of the class of supersmooth densities, cf.\ a relevant discussion in \cite{but}. Therefore without much loss of generality we assume that $s<2.$

With the above result we can finally study the asymptotics of $\hat{\rho}$ over the class of supersmooth L\'evy densities.

\begin{thm}
\label{thm-rho2} Suppose that conditions of Theorem \ref{thmsigmatilde2} are satisfied and let in addition Condition \ref{conditionkappa} hold. Then we have
\begin{equation*}
\suppp \operatorname{E}[(\hat{\rho}(x)-\rho(x))^2]\lesssim h^{s-1}\exp\left( -\frac{2\alpha}{h^s} \right)
\end{equation*}
for every fixed $x\neq 0.$ In particular, for $s=1$ an upper bound
\begin{equation}
\label{s1}
\suppp \operatorname{E}[(\hat{\rho}(x)-\rho(x))^2]\lesssim \exp\left( -{2\alpha} \left(\frac{\log n}{2\Sigma^2}\right)^{1/2}\right)
\end{equation}
is valid.
\end{thm}

Since $h\asymp (\log n)^{-1/2},$ which can be shown as formula (27) of \cite{but}, it is easy to see that the convergence rate of $\hat{\rho}$ is faster than any power of $\log n$ and hence much better than that in Theorem \ref{thm-rho}. The case $s=1$ is particularly interesting, as it corresponds to the class of L\'evy densities that admit an analytic continuation into a strip of the complex plane.

A natural question is whether $\hat{\rho}$ is rate-optimal over a class of supersmooth L\'evy densities. We will not provide a formal statement and its proof, but instead will restrict ourselves to an intuitive discussion, which we hope is more enlightening. To answer the question of rate-optimality of $\hat{\rho},$ one has first to establish a lower bound for estimation of $\rho(x)$ over a class of supersmooth L\'evy densities. Disregarding the moment condition in Condition \ref{conditionf2}, this can be done by following a general scheme of the proof of Theorem \ref{lowbound-thm} combined with some of the techniques from \cite{matias}, \cite{tsyb1} or  \cite{atomic}. This lower bound will be similar to the one given in Theorem 4 in \cite{tsyb1} and in fact for $s=1$ one will have
\begin{equation}
\label{lobnd}
\inf_{\widetilde{\rho}_n}\sup_{\mathcal{T}}\operatorname{E}[(\widetilde{\rho}_{n}(x)-\rho(x))^2]\gtrsim
\exp\left( -{2\alpha} \left(\frac{\log n}{\Sigma^2}\right)^{1/2}\right),
\end{equation}
where the infimum is taken over the class of all estimators $\widetilde{\rho}$ based on a sample $X_1,\ldots,X_n$ from the process $X.$ Unfortunately, the lower bound in \eqref{lobnd} is too small in comparison to the upper bound in \eqref{s1}. Although we are not completely sure, we still think that the lower and upper risk bounds that we give in Theorem \eqref{s1} and \eqref{lobnd} are sharp as far as their rates of decay are concerned: we think that it is the estimator $\hat{\rho}$ that cannot attain the minimax convergence rate. Given that this is true, an intuitive explanation of the suboptimality of $\hat{\rho}$ in the present setting might be the following: the construction of $\hat{\rho}$ in \eqref{fn} involves division by $(\hat{\phi}(t))^2,$ which is close to $(\phi_X(t))^2$ on $[-h^{-1},h^{-1}]$ for $n$ large enough. Hence in essence we are dealing with a kernel-type deconvolution density estimator which involves division by $(\phi_X(t))^2,$ whereas in conventional deconvolution problems the kernel estimator involves division by the characteristic function of the measurement error variable and not its square, see e.g.\ \cite{fan1}. By a rough analogy, assuming that the Gaussian component in the L\'evy process plays a role similar to the measurement error in the deconvolution problems, one can see that the variance of our estimator $\hat{\rho}$ of a L\'evy density is larger than the variance of a kernel-type deconvolution density estimator, compare p.\ 1266 in \cite{fan1} and an upper bound on the term $T_2$ in the proof of Theorem \ref{thm-rho2}. In order to render the variance asymptotically negligible, a somewhat larger bandwidth would thus be required in the former case than in the latter case. However, unlike the case when the L\'evy density satisfies Condition \ref{conditionf}, this has a dramatic effect on the bias of the estimator (as far as its order is concerned) for the class of supersmooth L\'evy densities and the suboptimality of $\hat{\rho}$ results: it is the squared bias, or roughly speaking the term $T_1$ in the proof of Theorem \ref{thm-rho2}, that dominates the asymptotics of $\hat{\rho}.$ No such problem seems to arise in \cite{comte}, where unlike our setting it is a priori assumed that $\gamma=0,\sigma=0,$ and as a consequence one can derive a different inversion formula than \eqref{inv}, cf.\ formula \eqref{**} below, which involves only division by $\phi_{X_1}$ and not by its square.

In light of the above observations another natural question that arises in this context is whether one has to use \eqref{ratio} instead of \eqref{inv} as a basis of construction of an estimator of $\rho$: under appropriate conditions with the former formula one can express the L\'evy density $\rho$ as
\begin{equation}
\label{**}
\rho(x)=-\frac{1}{2\pi x} \int_{\mathbb{R}} e^{-itx} \left( i \frac{\phi_{X_1}^{\prime}(t)}{\phi_{X_1}(t)} + \gamma + i \sigma^2 t  \right)dx,
\end{equation}
which involves division by the first power of $\phi_{X_1}$ only. By replacing $\phi_{X_1}$ by the empirical characteristic function $\hat{\phi}$ and $\sigma^2$ and $\gamma$ by their estimators and by application of an appropriate amount of regularisation we would thus get an estimator of $\rho$ that in its form is closer to a conventional kernel-type deconvolution density estimator in that under the integral sign it involves division by the first power of the (estimated) characteristic function only. It is nevertheless unclear whether this approach can lead to an estimator of $\rho$ with a better (optimal in the best case) convergence rate than the one we are considering in the present work: one has to find estimators of $\gamma$ and $\sigma^2$ that converge at an optimal rate in the present context, which does not seem to be an easy task.

Another interesting question that arises in the present context is that of adaptation: construction of our estimator of $\rho$ does rely on knowledge of the smoothness degree of a L\'evy density, see in particular Conditions \ref{conditionv}, \ref{conditionh2} and \ref{conditionv2}. In practice it might happen that this smoothness degree is unknown and it is desirable to have an estimator of $\rho$ that automatically achieves the optimal rate of convergence without knowledge of the smoothness degree of a L\'evy density. We view this as a separate problem and do not address it in the present work. Relevant results are available in the context of pure jump L\'evy processes and we refer e.g.\ to \cite{comte} for additional details. Note that the proofs of the adaptation results in that paper require nontrivial amount of technical work. In any case, in our setting an adaptive estimator and $\sigma^2$ would be required.

We conclude this section by a brief comparison of $\hat{\rho}$ to the
estimator $\rho_n$ of $\rho$ proposed in \cite{gug3}. Up to some additional
truncation, the latter estimator is given by
\begin{equation}
\label{deconvnoise-fnh}
\rho_{n}(x)=\frac{1}{2\pi}\int_{-1/h}^{1/h}
e^{-itx}\operatorname{Log}\left(\frac{\hat{\phi}(t)}{e^{i\hat{\gamma}t}e^{-\hat{\nu}(\mathbb{R})}e^{-\hat{\sigma}^2t^2/2}}\right)dt,
\end{equation}
where $\operatorname{Log}$ denotes the so-called distinguished
logarithm, i.e.\ a `logarithm' that is a continuous and single-valued function of $t,$
see Theorem 7.6.2 of \cite{chung} for its construction. Furthermore,
$\hat{\gamma},$ $\hat{\nu}(\mathbb{R})$ and $\hat{\sigma}^2$ are estimators
of the parameters $\gamma,\nu(\mathbb{R})$ and $\sigma^2,$ respectively. Notice that
in general the distinguished logarithm $\Log(g(t))$ of some function $g$ is not a composition of a fixed branch of
an ordinary logarithm with $g.$ The
estimator $\rho_{n}$ seems to be given by a more complicated
expression than $\hat{\rho},$ because it depends explicitly on
estimators of $\gamma$ and $\nu(\mathbb{R})$ in addition to the
estimator of $\sigma^2.$ The matter is furthermore complicated by the need to use the distinguished logarithm. The latter
in \eqref{deconvnoise-fnh} can be defined only for
those $\omega $'s from the sample space $\Omega$ for which
$\hat{\phi}$ as a function of $t$ does not hit zero on
$[-h^{-1},h^{-1}].$ For those $\omega$'s for which this is not
satisfied, $\rho_{n}$ has to be assigned an arbitrary value, e.g.\
one can assume that $\rho_{n}$ is a standard normal density. It is
shown in \cite{gug3} that as $n\rightarrow\infty,$ the probability of the event that
$\hat{\phi}$ hits zero for $t$ in $[-h^{-1},h^{-1}]$ vanishes under appropriate conditions. However, an almost sure result of a similar
type remains to be unknown (it has been established only in the context of \cite{belomestny} in \cite{sohl}). Also in practice the fact that $\hat{\phi}$ does not vanish can be checked for a discrete grid of points $t$ only and it could happen that one misses the fact that $\hat{\phi}(t)$ is zero for some $t\in[-h^{-1},h^{-1}].$ All this seems to be a disadvantage of the
estimator $\rho_{n}.$ On the other hand the estimator $\hat{\rho}$
is undefined for $x=0$ and a study of its asymptotic properties
requires stronger moment conditions on the L\'evy density $\rho.$ Also, a division by $x^{2}$ in the vicinity of the origin might render it numerically unstable. In conclusion, both estimators are rate-optimal over an appropriate class of L\'evy triplets, but each of them seems to have its own advantages over another.

\section{Proofs}
\label{proofs}

\begin{proof}[Proof of Theorem \ref{thm-ineq}]
The proof is similar in spirit to the one in \cite{gug3}, pp.\ 334--335, which in turn mimicks the one in \cite{matias}, pp.\ 326--327. Since both of the proofs are deficient, here we also seize an opportunity to rectify them.

We have
\begin{equation*}
\ex\left[\left(\sup_{t\in[-h^{-1},h^{-1}]}|\hat{\phi}^{(k)}(t)-\phi_{X_1}^{(k)}(t)|\right)^r\right]=\frac{1}{n^{r/2}}\ex\left[\left(\sup_{t\in[-h^{-1},h^{-1}]}|G_n v_{t,k}|\right)^r\right],
\end{equation*}
where $G_n v_{t,k}$ denotes an empirical process
\begin{equation*}
G_n
v_{t,k}=\frac{1}{\sqrt{n}}\sum_{j=1}^n(v_{t,k}(Z_j)-\operatorname{E}[v_{t,k}(Z_j)])
\end{equation*}
and the function $v_{t,k}$ is defined as $v_{t,k}:x\mapsto (ix)^k e^{itx}.$ Introduce the functions $v_{t,k,1}:x\mapsto x^k\sin(tx)$ and $v_{t,k,2}:x\mapsto x^k\cos(tx).$ Since $|i^k|=1$ and $e^{itx}=\cos(tx)+i\sin(tx),$ the $c_r$-inequality gives
\begin{align*}
\ex\left[\left(\sup_{t\in[-h^{-1},h^{-1}]}|G_n v_{t,k}|\right)^r\right]&\lesssim \ex\left[\left(\sup_{t\in[-h^{-1},h^{-1}]}|G_n v_{t,k,1}|\right)^r\right]\\
&+\ex\left[\left(\sup_{t\in[-h^{-1},h^{-1}]}|G_n v_{t,k,2}|\right)^r\right].
\end{align*}
Furthermore, by differentiability of $v_{t,k,j}$ with respect to $t$ and the mean-value theorem we have
\begin{equation}
\label{envelope}
|v_{t,k,j}(x)-v_{s,k,j}(x)|\leq |x|^{k+1}|t-s|
\end{equation}
for $j=1,2.$ Consequently, for a fixed $x$ the function $v_{t,k,j}$ is Lipschitz in $t$ with a
Lipschitz constant $|x|^{k+1}.$

In what follows we will need some results from the
theory of empirical processes. For all the unexplained terminology
and notation we refer e.g.\ to Section 19.2 of \cite{vdvaart} or
Section 2.1.1 of \cite{vaart}. First of all, by the inequality \eqref{envelope} and by Theorem 2.7.11 of \cite{vaart} the bracketing number $N_{[]}$ of
the class of functions $\mathbb{F}_{n,j}$ (for $j=1,2$ this refers to the collection of functions
$v_{t,k,j}$ for $t\in [-h^{-1},h^{-1}]$) can be bounded by the
covering number $N$ of the interval $I_n=[-h^{-1},h^{-1}]$ as
follows
\begin{equation*}
N_{[]}(2\epsilon\||x|^{k+1}\|_{\mathbb{L}_2(Q)};\mathbb{F}_{n,j};\mathbb{L}_2(Q))\leq N(\epsilon;I_n;|\cdot|).
\end{equation*}
Here $Q$ is any probability measure. Since it is easily seen that for the covering and bracketing numbers of the classes $\mathbb{F}_{n,j},$ $j=1,2,$ we have the inequality
\begin{equation*}
N(\epsilon\||x|^{k+1}\|_{\mathbb{L}_2(Q)};\mathbb{F}_{n,j};\mathbb{L}_2(Q))\leq N_{[]}(2\epsilon\||x|^{k+1}\|_{\mathbb{L}_2(Q)};\mathbb{F}_{n,j};\mathbb{L}_2(Q)),
\end{equation*}
cf.\ p.\ 84 in \cite{vaart}, and since
\begin{equation*}
N(\epsilon;I_n;|\cdot|)\leq \frac{1}{\epsilon}\frac{2}{h}+1,
\end{equation*}
we obtain that
\begin{equation}\
\label{ent1}
N(\epsilon\||x|^{k+1}\|_{\mathbb{L}_2(Q)};\mathbb{F}_{n,j};\mathbb{L}_2(Q))\leq \frac{1}{\epsilon}\frac{2}{h}+1.
\end{equation}
By taking $s=0,$ it follows from the definition of $v_{t,k,j}$ and \eqref{envelope} that the function
$F_{h,1}(x)=|x|^{k+1}h^{-1}$ can be used as an envelope for the
class $\mathbb{F}_{n,1},$ while
$F_{h,2}(x)=|x|^{k+1}h^{-1}+|x|^{k}$ can serve as an envelope for
$\mathbb{F}_{n,2}.$ Next define $J(1,\mathbb{F}_{n,j}),$ the
entropy of the class $\mathbb{F}_{n,j},$ as
\begin{equation*}
J(1,\mathbb{F}_{n,j})=\sup_Q\int_0^1 \{1+\log
(N(\epsilon\|F_{h,j}(x)\|_{\mathbb{L}_2(Q)};\mathbb{F}_{n,j};\mathbb{L}_2(Q)))
\}^{1/2}d\epsilon,
\end{equation*}
where $j=1,2,$ and the supremum is taken over all discrete probability measures
$Q,$ such that $\|F_{h,j}(x)\|_{\mathbb{L}_2(Q)}>0.$ Notice that
$\mathbb{F}_{n,j}$'s are measurable classes of functions with
measurable envelopes. Theorem 2.14.1 in \cite{vaart} then implies that
\begin{equation*}
\ex\left[\left(\sup_{t\in[-h^{-1},h^{-1}]}|G_n
v_{t,k,j}|\right)^r\right]\lesssim
\|F_{h,j}(x)\|^r_{\mathbb{L}_{2\vee r}(\operatorname{P})}(J(1,\mathbb{F}_{n,j}))^r.
\end{equation*}
Here the probability measure $\operatorname{P}$ on the righthand side is associated with the distribution of $X_1.$ We next need to work out the quantities on the righthand side of the above display. Observe that
\begin{equation*}
\|F_{h,1}(x)\|_{\mathbb{L}_{2\vee r}(\operatorname{P})}^r=\frac{1}{h^r} \||x|^{k+1}\|_{\mathbb{L}_{2\vee r}(\operatorname{P})}^r.
\end{equation*}
Moreover, we have
\begin{equation*}
\|F_{h,2}(x)\|_{\mathbb{L}_{2\vee r}(\operatorname{P})}^r\lesssim\frac{1}{h^r} (\||x|^{k+1}\|_{\mathbb{L}_{2\vee r}(\operatorname{P})}^r + \||x|^{k}\|_{\mathbb{L}_{2\vee r}(\operatorname{P})}^r),
\end{equation*}
provided $h\leq 1.$ Here we also used the $c_{2\vee r}$-inequality. It thus remains to bound the entropy $J(1,\mathbb{F}_{n,j}).$ By the fact that
\begin{equation*}
\|F_{h,1}(x)\|_{\mathbb{L}_2(Q)}=h^{-1}\||x|^{k+1}\|_{\mathbb{L}_2(Q)}
\end{equation*}
and by taking $\epsilon / h$ instead of $\epsilon$ in \eqref{ent1} we get
\begin{equation}
\label{n1}
N(\epsilon\|F_{h,1}(x)\|_{\mathbb{L}_2(Q)};\mathbb{F}_{n,j};\mathbb{L}_2(Q))\leq \frac{2}{\epsilon}+1.
\end{equation}
Furthermore, since $\|F_{h,2}(x)\|_{\mathbb{L}_2(Q)}\geq
\||x|^{k+1}h^{-1}\|_{\mathbb{L}_2(Q)},$ by monotonicity of the covering number
$N$ in the size of the covering balls combined with \eqref{n1} we
obtain that
\begin{equation}
\label{n2}
N(\epsilon\|F_{h,2}(x)\|_{\mathbb{L}_2(Q)};\mathbb{F}_{n,j};\mathbb{L}_2(Q))\leq \frac{2}{\epsilon}+1.
\end{equation}
Inserting the bounds from \eqref{n1} and \eqref{n2} into the definition of $J(1,\mathbb{F}_{n,j})$, we see that
\begin{equation*}
J(1,\mathbb{F}_{n,j})\leq \int_0^1\left\{1+\log \left(\frac{2}{\epsilon}+1\right)\right\}^{1/2}d\epsilon<\infty.
\end{equation*}
This yields the statement of the theorem.
\end{proof}

\begin{proof}[Proof of Theorem \ref{thm-rho}]
By the $c_2$-inequality we have
\begin{equation*}
\operatorname{E}[(\hat{\rho}(x)-\rho(x))^2]\lesssim |\rho(x)-\widetilde{\rho}(x)|^2+\operatorname{E}[|\hat{\rho}(x)-\widetilde{\rho}(x)|^2]=T_1+T_2,
\end{equation*}
where
\begin{equation*}
\widetilde{\rho}(x)=\frac{1}{2\pi x^2}\int_{-1/h}^{1/h} e^{-itx}\left(\frac{(\phi_{X_1}^{\prime}(t))^2-\phi_{X_1}^{\prime\prime}(t)\phi_{X_1}(t)}{(\phi_{X_1}(t))^2}-\sigma^2\right)dt.
\end{equation*}
We will first work out the term $T_1.$ By \eqref{ast*} we have
\begin{equation*}
-\phi_{\rho}^{\prime\prime}(t)=\frac{(\phi_{X_1}^{\prime}(t))^2-\phi_{X_1}^{\prime\prime}(t)\phi_{X_1}(t)}{(\phi_{X_1}(t))^2}-\sigma^2.
\end{equation*}
Then by the Fourier inversion argument we can write
\begin{equation*}
\rho(x)-\widetilde{\rho}(x)=\frac{1}{2\pi}\int_{\mathbb{R}} e^{-itx} \phi_{\rho}(t)dt + \frac{1}{2\pi x^2} \int_{-1/h}^{1/h} e^{-itx} \phi_{\rho}^{\prime\prime}(t)dt.
\end{equation*}
Integrating by parts twice the second term on the righthand side of the above display and using Condition \ref{conditionf}, we obtain
\begin{equation*}
\frac{1}{2\pi x^2} \int_{-1/h}^{1/h} e^{-itx} \phi_{\rho}^{\prime\prime}(t)dt = - \frac{1}{2\pi} \int_{-1/h}^{1/h} e^{-itx} \phi_{\rho}(x)dx + O( h^{\beta} ),
\end{equation*}
where the $O(h^{\beta})$ term on the righthand side is uniform in $\rho.$ With this in mind and by the fact that $\phi_{\rho}(t)=\nu(\mathbb{R}) \phi_f(t),$ we can bound $T_1$ using the $c_2$-inequality as
\begin{align*}
T_1 & \lesssim \frac{\Lambda^2}{4\pi^2}\left( \int_{\mathbb{R}\setminus[-h^{-1},h^{-1}]} |\phi_f(t)|dt\right)^2 + h^{2\beta}\\
& \lesssim \left( \int_{\mathbb{R}\setminus[-h^{-1},h^{-1}]} |t|^{\beta}|t|^{-\beta}|\phi_f(t)|dt\right)^2  + h^{2\beta}\\
& \leq \left( \int_{-\infty}^{\infty} |t|^{\beta}|\phi_f(t)|dt\right)^2 h^{2\beta}  + h^{2\beta}\\
& \lesssim h^{2\beta},
\end{align*}
provided that $h\leq 1.$ Hence by Condition \ref{conditionsigma} the term
$\suppp T_1$ is of order $(\log n)^{-\beta}.$ This is the term that has the dominating contribution to the risk of $\hat{\rho}.$ The rest of the proof is dedicated to showing that $T_2$ is negligible in comparison to $T_1.$ This involves a long series of inequalities.

By the $c_2$-inequality we have
\begin{align*}
T_2 & \lesssim \frac{1}{4\pi^2x^4} \left| \int_{-1/h}^{1/h}e^{-itx} dt \right|^2 \ex[|\hat{\sigma}^2-\sigma^2|^2]\\
& + \frac{1}{4\pi^2x^4}\ex\left[ \left| \int_{-1/h}^{1/h} e^{-itx} (\Phi(\hat{\phi}(t))1_{G_t}  - \Phi({\phi}(t)) )dt \right|^2    \right]\\
&=T_3+T_4,
\end{align*}
where for a twice differentiable function $\zeta$ the mapping $\Phi$ is defined by
\begin{equation*}
\Phi(\zeta(t))=\frac{(\zeta^{\prime}(t))^2-\zeta^{\prime\prime}(t)\zeta(t)}{(\zeta(t))^2}.
\end{equation*}
By Theorem 2.1 in combination with Condition \ref{conditionh} we
have $\suppp T_3\lesssim (\log n)^{-\beta-2}.$ Next notice that
\begin{align*}
T_4 & \leq \frac{1}{\pi^2x^4 h^2}\ex\left[ \left(  \sup_{t\in[-h^{-1},h^{-1}]} |\Phi(\hat{\phi}(t))1_{G_t}  - \Phi({\phi}(t))|  \right)^2 \right]\\
& = \frac{T_5}{\pi^2x^4}.
\end{align*}
Hence it remains to study $T_5.$ This will be done via repeated applications of Theorem \ref{thm-ineq}. First of all, the $c_2$-inequality gives
\begin{align*}
T_5 & \lesssim \frac{1}{h^2} \ex\left[ \left( \supp\left|  \frac{\hat{\phi}^{\prime\prime}(t)}{\hat{\phi}(t)}1_{G_t} - \frac{{\phi}^{\prime\prime}_{X_1}(t)}{\phi_{X_1}(t)} \right|       \right)^2                 \right]\\
&+\frac{1}{h^2}\ex\left[ \left( \supp\left|  \frac{(\hat{\phi}^{\prime}(t))^2}{(\hat{\phi}(t))^2}1_{G_t} - \frac{({\phi}^{\prime}_{X_1}(t))^2}{(\phi_{X_1}(t))^2} \right|       \right)^2                 \right]\\
&=T_6+T_7.
\end{align*}
By another application of the $c_2$-inequality we obtain
\begin{align*}
T_6 & \lesssim \frac{1}{h^2}\ex\left[ \left( \supp\left|  \frac{\hat{\phi}^{\prime\prime}(t)}{\hat{\phi}(t)}1_{G_t} - \frac{{\phi}^{\prime\prime}_{X_1}(t)}{\phi_{X_1}(t)}1_{G_t} \right|       \right)^2                 \right]\\
&+\frac{1}{h^2}\ex\left[ \left(\supp \left(\left| \frac{\phi_{X_1}^{\prime\prime}(t)}{\phi_{X_1}(t)} \right| 1_{G_t^c}\right)\right)^2 \right]\\
&=T_8+T_9.
\end{align*}
The term $T_8$ in the last equality can be bounded as follows,
\begin{align*}
T_8 & \lesssim \frac{1}{h^2}\ex\left[ \left( \supp\left|  \frac{\hat{\phi}^{\prime\prime}(t)}{\hat{\phi}(t)}1_{G_t} - \frac{\hat{\phi}^{\prime\prime}(t)}{\phi_{X_1}(t)}1_{G_t} \right|       \right)^2                 \right]\\
& + \frac{1}{h^2}\ex\left[ \left( \supp\left|  \frac{\hat{\phi}^{\prime\prime}(t)}{{\phi}_{X_1}(t)}1_{G_t} - \frac{{\phi}^{\prime\prime}_{X_1}(t)}{\phi_{X_1}(t)}1_{G_t} \right|       \right)^2                 \right]\\
& = T_{10}+T_{11}.
\end{align*}
Further bounding gives
\begin{align*}
T_{10} & \leq \frac{1}{h^2}\ex\left[\left( \supp
|\hat{\phi}^{\prime\prime}(t)| \right)^2 \left( \supp
\left(\frac{|\hat{\phi}(t)-{\phi}_{X_1}(t)|}{|\hat{\phi}(t)||{\phi}_{X_1}(t)|}1_{G_t}\right)
\right)^2     \right].
\end{align*}
Now apply the Cauchy-Schwarz inequality to the righthand side to obtain
\begin{align*}
T_{10} & \leq \frac{1}{h^2}\left(\ex\left[\left( \supp |\hat{\phi}^{\prime\prime}(t)| \right)^4\right]\right)^{1/2}\\
& \times
\left(\ex\left[\left( \supp \left(\frac{|\hat{\phi}(t)-{\phi}_{X_1}(t)|}{|\hat{\phi}(t)||{\phi}_{X_1}(t)|}1_{G_t}\right)     \right)^4     \right]\right)^{1/2}\\
&=\frac{1}{h^2}\sqrt{T_{12}}\sqrt{T_{13}}.
\end{align*}
Observe that by the fact that $|\hat{\phi}^{\prime\prime}(t)|\leq n^{-1}\sum_{j=1}^n Z_j^2$ and by the $c_4$-inequality
\begin{align*}
T_{12} & \leq \ex\left[ \left| \frac{1}{n}\sum_{j=1}^n Z_j^2 \right|^4 \right]\\
& \leq \frac{c_4}{n^4}\ex\left[ \left|\sum_{j=1}^n (Z_j^2-\ex[Z_j^2])\right|^4 \right]+c_4 (\ex[Z_1^2])^4\\
& \leq (3\sqrt{2})^4 4^{4/2} \frac{c_4}{n^2}\ex[(Z_1^2-\ex[Z_1^2])^4]+c_4(\ex[Z_1^2])^4,
\end{align*}
where the last inequality follows from the Marcinkiewicz-Zygmund
inequality as given in Theorem 2 of \cite{ren}. By the Lyapunov
inequality $(\ex[Z_1^2])^4\leq \ex[Z_1^8].$ This in combination
with the $c_4$-inequality gives $\ex[(Z_1^2-\ex[Z_1^2])^4]\lesssim
\ex[Z_1^8].$ It remains to bound $\ex[Z_1^8]$ uniformly in L\'evy
triplets. The most direct way of doing this is to notice that
\begin{equation*}
\ex[Z_1^8]=\ex[(\gamma+\sigma W + Y)^8]\lesssim \Gamma^8+\Sigma^8 \ex[W^8]+\ex[Y^8],
\end{equation*}
where $W$ is a standard normal random variable, while $Y$ has a
compound Poisson distribution with intensity $\nu(\mathbb{R})$ and
jump size density $f.$ Observe that $\ex[Y^8]=\phi_Y^{(8)}(0)$
and that under Condition \ref{conditionf} and with the Lyapunov inequality it is laborious, though straightforward to show that $\phi_Y^{(8)}(0)$ is bounded  by a universal
constant uniformly in L\'evy triplets. Hence the term $\suppp \ex[Z_1^8]$ is also bounded and then so is
$\suppp \sqrt{T_{12}}.$ As far as $T_{13}$ is concerned, we have
\begin{equation*}
T_{13}\lesssim \frac{e^{4\Sigma^2/h^2}}{\kappa_n^4}\ex\left[ \left( \supp |\hat{\phi}(t)-\phi_{X_1}(t)| \right)^4 \right],
\end{equation*}
which follows from Conditions \ref{conditionf} and \ref{conditionsigma}. Inequality \eqref{ineq-1} with $k=0$ and $r=4$ then yields
\begin{equation*}
T_{13}\lesssim
\||x|\|_{\mathbb{L}_{4}(\operatorname{P})}^4\frac{e^{4\Sigma^2/h^2}}{\kappa_n^4 h^4n^{2}}.
\end{equation*}
Since $\||x|\|_{\mathbb{L}_{4}(\operatorname{P})}$ is
bounded by a constant uniformly in L\'evy triplets (this can be proved by
essentially the same argument as we used for $\suppp \ex[Z_1^8]$
above), it follows that $\suppp T_{13}$ is negligible in
comparison to $(\log n)^{-\beta}.$ This is also true for $h^{-2}\suppp \sqrt{T_{13}}$ and then also for $\suppp T_{10}.$ To complete the study of $T_8,$ we
need to study $T_{11}.$ The latter can be bounded as follows:
\begin{equation*}
T_{11}\lesssim \frac{e^{\Sigma^2/h^2}}{h^2}\ex\left[ \left( \supp
|\hat{\phi}^{\prime\prime}(t)-\phi_{X_1}^{\prime\prime}(t)|
\right)^2 \right].
\end{equation*}
By the same reasoning as above one can show that $\suppp T_{11}$ is negligible
compared to $(\log n)^{-\beta}.$ Consequently, so is $\suppp T_8.$
Next we deal with $T_9.$ Notice that by our conditions and the Lyapunov inequality
\begin{align*}
\left|\frac{\phi_{X_1}^{\prime\prime}(t)}{\phi_{X_1}(t)}\right| & \leq \left|\frac{\phi_{X_1}^{\prime}(t)}{\phi_{X_1}(t)}\right|^2+\sigma^2+\infint x^2\rho(x)dx\\
& \leq \left(\Gamma+\Sigma^2\frac{1}{h}+\Lambda K^{1/12}\right)^2+\Sigma^2+\Lambda K^{1/6}\\
& \lesssim \frac{1}{h^2}.
\end{align*}
Hence it holds that
\begin{equation}
\label{sder}
\suppp \supp\left|\frac{\phi_{X_1}^{\prime\prime}(t)}{\phi_{X_1}(t)}\right|\lesssim \frac{1}{h^2}.
\end{equation}
Consequently, we have
\begin{equation*}
T_9\lesssim \frac{1}{h^4}\ex\left[ \left(\supp 1_{G_t^c} \right)^2 \right].
\end{equation*}
We study the expectation on the righthand side. First of all, for $t\in[-h^{-1},h^{-1}]$ and all $n$ large enough we have
\begin{align*}
G_t^c & = \left\{ |\hat{\phi}(t)|-|\phi_{X_1}(t)|<\kappa_n e^{-\Sigma^2/(2h^2)}-|\phi_{X_1}(t)| \right\}\\
& = \left\{ |\phi_{X_1}(t)|-|\hat{\phi}(t)| > |\phi_{X_1}(t)| - \kappa_n e^{-\Sigma^2/(2h^2)} \right\}\\
& \subseteq \left\{ |\phi_{X_1}(t)-\hat{\phi}(t)| > (e^{-2\Lambda} - \kappa_n) e^{-\Sigma^2/(2h^2)} \right\}\\
& \subseteq \left\{ \supp |\phi_{X_1}(t)-\hat{\phi}(t)| > (e^{-2\Lambda} - \kappa_n) e^{-\Sigma^2/(2h^2)} \right\}\\
& = G^*.
\end{align*}
Therefore $\sup_{t\in[-h^{-1},h^{-1}]}1_{G_t^c}\leq 1_{G^*}$ and then by Chebyshev's inequality we obtain
\begin{equation}
\label{T9bound}
T_9 \lesssim \frac{1}{h^4} \operatorname{P}(G^*) \lesssim
\frac{e^{2\Sigma^2/h^2}}{h^4} \ex\left[\left( \supp
|\phi_{X_1}(t)-\hat{\phi}(t)|\right)^4 \right].
\end{equation}
Next apply \eqref{ineq-1} with $k=0$ and $r=4$ to the
expectation in the rightmost inequality to conclude that
$\sup_{\mathcal{T}}T_9$ is negligible in comparison to $(\log
n)^{-\beta}.$ This shows that also $\suppp{T_6}$ is negligible in
comparison to $(\log n)^{-\beta}.$ To complete bounding $T_5$ and
eventually $T_4,$ we need to bound $T_7.$ By the $c_2$-inequality
\begin{align*}
T_7 & \lesssim \ex\left[ \left( \supp \left( \left| \frac{({\phi}^{\prime}_{X_1}(t))^2}{(\phi_{X_1}(t))^2}\right| 1_{G_t^c} \right)       \right)^2                 \right]\\
& + \ex\left[ \left( \supp\left|  \frac{(\hat{\phi}^{\prime}(t))^2}{(\hat{\phi}(t))^2}1_{G_t} - \frac{({\phi}^{\prime}_{X_1}(t))^2}{(\phi_{X_1}(t))^2}1_{G_t} \right|       \right)^2                 \right]\\
& = T_{14}+T_{15}.
\end{align*}
Observe that for $h\rightarrow 0$ we have
\begin{equation*}
\suppp \sup_{t\in[-h^{-1},h^{-1}]} \left| \frac{({\phi}^{\prime}_{X_1}(t))^2}{(\phi_{X_1}(t))^2}\right| \lesssim \frac{1}{h^2},
\end{equation*}
which can be shown by the same arguments that led to \eqref{sder}. We also have $T_{14} \leq h^{-4} \operatorname{P}(G^*)$ by the above display. It then follows from \eqref{T9bound} that $\suppp T_{14}$ is negligible in comparison to $(\log n)^{-\beta}.$ We turn to $T_{15}.$ By the $c_2$-inequality
\begin{align*}
T_{15} & \lesssim \ex\left[ \left( \supp\left|  \frac{(\hat{\phi}^{\prime}(t))^2}{(\hat{\phi}(t))^2}1_{G_t} - \frac{(\hat{\phi}^{\prime}(t))^2}{(\phi_{X_1}(t))^2}1_{G_t} \right|       \right)^2                 \right]\\
& + \ex\left[ \left( \supp\left|  \frac{(\hat{\phi}^{\prime}(t))^2}{({\phi}_{X_1}(t))^2}1_{G_t} - \frac{({\phi}^{\prime}_{X_1}(t))^2}{(\phi_{X_1}(t))^2}1_{G_t} \right|       \right)^2                 \right]\\
& = T_{16}+T_{17}.
\end{align*}
Notice that by the Cauchy-Schwarz inequality
\begin{align*}
T_{16} & \leq \ex\left[ \left( \supp|  {(\hat{\phi}^{\prime}(t))^2}|\supp\left(\left|\frac{1}{(\hat{\phi}(t))^2} - \frac{1}{(\phi_{X_1}(t))^2}\right|1_{G_t}\right) \right)^2 \right]\\
& = \ex\left[ \left( \supp|  {(\hat{\phi}^{\prime}(t))^2}|\supp\left(\frac
{|(\phi_{X_1}(t))^2-(\hat{\phi}(t))^2|}{|\hat{\phi}(t)|^2|\phi_{X_1}(t)|^2}
1_{G_t}\right) \right)^2 \right]\\
& \leq \left(
\ex\left[ \left( \supp|  {(\hat{\phi}^{\prime}(t))^2}|\right)^4 \right] \right)^{1/2}\\
& \times
\left(\ex\left[ \left( \supp\left(\frac
{|(\phi_{X_1}(t))^2-(\hat{\phi}(t))^2|}{|\hat{\phi}(t)|^2|\phi_{X_1}(t)|^2}1_{G_t}\right) \right)^4 \right]\right)^{1/2}\\
& =\sqrt{T_{18}}\sqrt{T_{19}}.
\end{align*}
Since $|\hat{\phi}^{\prime}(t)|\leq
n^{-1}\sum_{j=1}^n |Z_j|,$ it follows that the term $T_{18}$ is bounded by
$\ex[(n^{-1}\sum_{j=1}^n |Z_j|)^8].$ By the $c_8$-inequality we
then get
\begin{equation*}
\ex\left[\left(\frac{1}{n}\sum_{j=1}^n |Z_j|\right)^8\right] \lesssim \frac{1}{n^8}\ex\left[\left|\sum_{j=1}^n (|Z_j| - \ex[|Z_j|])\right|^8\right] + (\ex[|Z_j|])^8.
\end{equation*}
Hence $\suppp T_{18}$ is bounded by a constant, which can be proved by
the same argument as we used for $\suppp T_{12}.$ Finally, we consider
$T_{19}.$ We have
\begin{equation*}
T_{19}  \lesssim \frac{e^{4\Sigma^2/h^2}}{k_n^8} \ex\left[ \left(
\supp |\hat{\phi}(t)-\phi_{X_1}(t)| \right)^4 \right],
\end{equation*}
because
\begin{equation*}
|(\phi_{X_1}(t))^2-(\hat{\phi}(t))^2| \leq 2 |\phi_{X_1}(t)-\hat{\phi}(t)|,
\end{equation*}
because $|\phi_{X_1}(t)|$ is bounded from below by
$e^{-2\Lambda-\Sigma^2/(2h^2)}$ for $t\in[-h^{-1},h^{-1}],$ and because of the definition of $G_t.$ Using
\eqref{ineq-1}, we conclude that $\suppp T_{19}$ is negligible in
comparison to $(\log n)^{-\beta}.$ Hence so is $\suppp T_{16}.$ It
remains to study $T_{17}.$ Since
\begin{equation*}
T_{17}\lesssim e^{2\Sigma^2/h^2}\ex\left[ \left( \supp |\hat{\phi}^{\prime}(t)-\phi_{X_1}^{\prime}(t)| \right)^2 \right],
\end{equation*}
it follows from \eqref{ineq-1} and Condition \ref{conditionh} that $\suppp T_{17}$ is negligible
in comparison to $(\log n)^{-\beta}.$ Consequently, so are $\suppp
T_{15}$ and $\suppp T_7.$ Combination of all the above results
completes the proof of the theorem.
\end{proof}

\begin{proof}[Proof of Theorem \ref{lowbound-thm} ]
The statement of the theorem is for estimators based on
observations $X_1,\ldots,X_n,$ but the relationship
$Z_j=X_j-X_{j-1}$ and the stationary independent increments
property of a L\'evy process allows us to work with
$Z_1,\ldots,Z_n$ instead. We adapt the proof of Theorem 4.1 in
\cite{gug3} to the present case. A general idea of the proof is as
follows: we will consider two L\'evy triplets
$T_1=(0,\sigma^2,\rho_1)$ and $T_2=(0,\sigma^2,\rho_2)$ depending
on $n$ and such that the L\'evy densities $\rho_1$ and $\rho_2$
are separated as much as possible at a point $x,$ while at the same time the
corresponding product densities $q_1^{\otimes n}$ and
$q_2^{\otimes n}$ of observations $Z_1,\ldots,Z_n$ are close in
the $\chi^2$-divergence and hence cannot be distinguished well
using the observations $Z_1,\ldots,Z_n.$ Up to a constant, the squared distance
between $\rho_1(x)$ and $\rho_2(x)$ will then
give the desired lower bound \eqref{lowerbound} for estimation of
a L\'evy density $\rho$ at a fixed point $x.$ This is a standard
technique and we refer to Chapter 2 of \cite{tsyb} for a good
exposition of methods for deriving lower bounds in nonparametric
curve estimation.

Consider two L\'evy triplets $T_1=(0,\sigma^2,\rho_1)$ and
$T_2=(0,\sigma^2,\rho_2),$ where $\rho_j(u)=\nu(\mathbb{R})
f_j(u)$ for $j=1,2$ and constants $0<\nu(\mathbb{R})<\Lambda$ and $0<\sigma^2<\Sigma^2.$
Let
\begin{equation*}
f_1(u)=\frac{1}{2}(r_1(u)+r_2(u)),
\end{equation*}
where two densities $r_1$ and $r_2$ are defined through their characteristic functions as follows:
\begin{gather*}
r_1(u)=\frac{1}{2\pi}\infint
e^{-itu}\frac{1}{(1+t^2/\beta_1^2)^{(\beta_2+1)/2}}dt,\\
r_2(u)=\frac{1}{2\pi}\infint
e^{-itu}e^{-\alpha_1|t|^{\alpha_2}}dt.
\end{gather*}
With a proper selection of $\beta_1,\beta_2,\alpha_1$ and
$\alpha_2$ one can achieve that $f_1$ satisfies \eqref{fcnd} with
constants $L/2,$ $L^{\prime}/2$ and $L^{\prime\prime}/2$ instead of $L,$ $L^{\prime}$ and $L^{\prime\prime}.$ We also assume that
$1<\alpha_2<2.$ Next define $f_2$ by
\begin{equation*}
f_2(u)=f_1(u)+\delta_n^{\beta}H((u-x)/\delta_n),
\end{equation*}
where $\delta_n\rightarrow 0$ as $n\rightarrow\infty,$ and the
function $H$ satisfies the following conditions:
\begin{enumerate}
\item $H(0)>0;$
\item $\phi_H(t)$ is twice continuously differentiable;
\item $\infint |t|^{\beta}|\phi_H(t)|dt\leq L/2, \quad |\phi_H(t)|\leq {L^{\prime}}/{(2|t|^{\beta})}, \quad |\phi_H^{\prime}(t)| \leq {L^{\prime\prime}}/{(2|t|^{\beta})};$
\item $\infint H(x)dx=0;$
\item $\int_{-\infty}^0 H(x)dx\neq 0;$
\item $\phi_H(t)=0$ for $t$ outside $[1,2].$
\end{enumerate}
Since $f_1(u)$ decays as $r_2(u)$ at infinity, and consequently as $|u|^{-1-\alpha_2},$
see formula (14.37) in \cite{sato}, with a proper selection of $H,$ e.g.\ by the reasoning similar to the one on p.\ 1268 in \cite{fan1}, the function $f_2$ will be nonnegative, at least for all small enough $\delta_n.$ Consequently, $f_2$ will be a probability density and one can also achieve that it satisfies \eqref{fcnd} for all small enough $\delta_n.$

Now notice that
\begin{equation}
\label{distance} |\rho_2(x)-\rho_1(x)|^2\asymp \delta_n^{2\beta}.
\end{equation}
The statement of the theorem will follow from \eqref{distance} and Lemma 8 of \cite{tsyb1}, if we
prove that for $\delta_n\asymp (\log n)^{-1/2}$ we have
\begin{equation}
\label{chi-square} n\chi^2(q_2,q_1)=n\infint
\frac{(q_2(u)-q_1(u))^2}{q_1(u)}du\leq c,
\end{equation}
where a positive constant $c<1$ is independent of $n.$ Here $\chi^2(\cdot,\cdot)$ denotes the $\chi^2$-divergence, see p.\ 86 in \cite{tsyb} for
the definition.

Denote by $p_i$ a density of a Poisson sum
$Y=\sum_{j=1}^{N(\nu(\mathbb{R}))}W_j$ conditional on the fact
that its number of summands $N(\nu(\mathbb{R}))>0.$ Here $W_j$ are
i.i.d.\ with $W_1\sim f_i.$ Now rewrite the characteristic
function of $Y$ as
\begin{equation}
\label{ychf}
\phi_Y(t)=e^{-\nu(\mathbb{R})}+(1-e^{-\nu(\mathbb{R})})\frac{1}{e^{\nu(\mathbb{R})}-1}\left(e^{\nu(\mathbb{R})\phi_{f_i}(t)}-1\right),
\end{equation}
to see that
\begin{equation*}
\phi_{p_i}(t)=\frac{1}{e^{\nu(\mathbb{R})}-1}\left(e^{\nu(\mathbb{R})\phi_{f_i}(t)}-1\right).
\end{equation*}
Furthermore,
\begin{equation}\label{g-expr}
p_i(u)=\sum_{n=1}^{\infty}f_i^{\ast
n}(u)P(N(\nu(\mathbb{R}))=n|N(\nu(\mathbb{R}))>0).
\end{equation}
By convolving the law of $Y$ with a normal density $\phi_{0,\sigma^2}$ with mean zero and variance $\sigma^2$ and using \eqref{ychf}, we obtain that
\begin{equation*}
q_1(u)\geq (1-e^{-\nu(\mathbb{R})})\phi_{0,\sigma^2}\ast p_1(u).
\end{equation*}
Since by Lemma 2 of \cite{matias} there exists a large enough constant
$A,$ such that the right-hand side of the above display is not
less than $(1-e^{-\nu(\mathbb{R})})p_1(|u|+A),$ we have
\begin{equation*}
n\chi^2(q_2,q_1)\lesssim n\infint
\frac{(q_2(u)-q_1(u))^2}{p_1(|u|+A)}dx\lesssim n\infint
\frac{(q_2(u)-q_1(u))^2}{f_1(|u|+A)}dx.
\end{equation*}
The last inequality is true because by \eqref{g-expr} it holds that $p_1(|u|+A)\gtrsim f_1(|u|+A).$
Splitting the integration region in the rightmost term of the last
display into two parts, we get that
\begin{align*}
n\chi^2(q_2,q_1)&\lesssim n\int_{|u|\leq A} {(q_2(u)-q_1(u))^2}du+n\int_{|u|>A} u^4{(q_2(u)-q_1(u))^2}dx\\
&=T_1+T_2.
\end{align*}
Here we used the facts that $f_1(u)$ decays as $|u|^{-1-\alpha_2}$
at infinity and that
$1<\alpha_2<2.$ Parseval's identity then gives
\begin{align*}
T_1&\leq n\frac{1}{2\pi}\infint
|\phi_{q_2}(t)-\phi_{q_1}(t)|^2dt\\
&=n\frac{(1-e^{-\nu(\mathbb{R})})^2}{2\pi}\infint
|\phi_{p_2}(t)-\phi_{p_1}(t)|^2e^{-\sigma^2t^2}dt\\
&=n\frac{(1-e^{-\nu(\mathbb{R})})^2}{(e^{\nu(\mathbb{R})}-1)^2}\frac{1}{2\pi}\infint
|e^{\nu(\mathbb{R})\phi_{f_2}(t)}-e^{\nu(\mathbb{R})\phi_{f_1}(t)}|^2e^{-\sigma^2t^2}dt\\
& \lesssim n\infint
|\phi_{f_2}(t)-\phi_{f_1}(t)|^2e^{-\sigma^2t^2}dt,
\end{align*}
where the last inequality is a consequence of the mean-value theorem
applied to the function $e^x$ and the fact that
$|\nu(\mathbb{R})\phi_{f_i}(t)|\leq\Lambda<\infty.$  Now notice that
\begin{equation*}
\infint
e^{itu}\delta_n^{\beta}H((u-x)/\delta_n)dx=\delta_n^{\beta+1}e^{itx}\phi_H(\delta_n
t).
\end{equation*}
By definition of $f_1$ and
$f_2$ it follows that
\begin{align*}
T_1&\lesssim n\delta_n^{2\beta+2}\infint |\phi_H(\delta_n
t)|^2e^{-\sigma^2t^2}dt\\
&=n\delta_n^{2\beta+1}\infint
|\phi_H(s)|^2e^{-\sigma^2s^2/\delta_n^2}ds\\
&=O\left(n\delta_n^{2\beta+1}e^{-\sigma^2/\delta_n^2}\right).
\end{align*}
Hence a choice $\delta_n\asymp (\log n)^{-1/2}$ with an appropriate constant will imply that $T_1\rightarrow 0$ as
$n\rightarrow\infty.$

To complete the proof, we need to show that $T_2\rightarrow 0$
under a suitable condition on $\delta_n.$ To this end first notice that even though $\phi_{f_1}$ and $\phi_{f_2}$ are not
twice differentiable at zero, the difference
$\phi_{q_2}(t)-\phi_{q_1}(t)$ still is, because $\phi_H$ is
identically zero outside the interval $[1,2],$ and hence $\phi_{q_2}(t)-\phi_{q_1}(t)$ is zero for $t$ in a neighbourhood of zero. Then by Parseval's identity we obtain that
\begin{equation*}
T_2\leq n\frac{1}{2\pi}\infint
|(\phi_{q_2}(t)-\phi_{q_1}(t))''|^2dt.
\end{equation*}
By the same arguments as we used for $T_1,$ one can show that $T_2\rightarrow
0$ as $n\rightarrow\infty,$ provided $\delta_n\asymp (\log
n)^{-1/2}$ with an appropriate constant. This entails the
statement of the theorem.
\end{proof}

The following technical lemma is used in the proof of Theorem \ref{thmsigmatilde2}.

\begin{lem}
\label{lemmaa1}
Let the sets $B_n$ and $B_n^c$ be defined as
\begin{equation}
\label{bnset}
\begin{split}
B_{n}&=\left\{\sup_{t\in[-\sqrt{2}h^{-1},\sqrt{2}h^{-1}]}\left|{\hat{\phi}(t)}-{\phi_{X_1}(t)}\right|>\delta\right\},\\
B_{n}^c&=\left\{\sup_{t\in[-\sqrt{2}h^{-1},\sqrt{2}h^{-1}]}\left|{\hat{\phi}(t)}-{\phi_{X_1}(t)}\right|\leq\delta\right\},
\end{split}
\end{equation}
where $\delta=(1/4)e^{-2\Lambda-\Sigma^2/h^2}.$
Suppose that $\nu(\mathbb{R})\leq\Lambda<\infty$ and that Conditions \ref{conditionsigma}, \ref{conditiongamma}, \ref{conditionm} and \ref{conditionh2} hold.
Then there exists a universal $n_0$ not depending on the L\'evy triplet $(\gamma,\sigma,\rho),$ such that for all $n\geq n_0$ on the set $B_n^c$ we have
\begin{equation*}
\max\{ \min\{ M_n,\log(|\hat{\phi}(t)|)  \},-M_n \}=\log(|\hat{\phi}(t)|)
\end{equation*}
for $t$ restricted to the interval $[-\sqrt{2}h^{-1},\sqrt{2}h^{-1}].$
\end{lem}
\begin{proof}
The proof is similar to the proof of Lemma 5.1 in \cite{gug3}. On the set $B_n^c$ and for $t$ restricted to the interval $[-\sqrt{2}h^{-1},\sqrt{2}h^{-1}]$ we have
\begin{equation}
\label{logineqhalf}
\left|\left|\frac{\hat{\phi}(t)}{\phi_{X_1}(t)}\right|-1\right|\leq\left|\frac{\hat{\phi}(t)}{\phi_{X_1}(t)}-1\right|<\frac{1}{2}.
\end{equation}
Furthermore, on the same set and for $t\in[-\sqrt{2}h^{-1},\sqrt{2}h^{-1}]$ the inequality
\begin{align*}
|\log(|\hat{\phi}(t))||&\leq |\log(|\phi_{X_1}(t)|)|+\left|\log\left(\left|\frac{\hat{\phi}(t)}{\phi_{X_1}(t)}\right|\right)\right|\\
&\leq |\log(|\phi_{X_1}(t)|)|+\left|\frac{\hat{\phi}(t)}{\phi_{X_1}(t)}-1\right|+\left|\frac{\hat{\phi}(t)}{\phi_{X_1}(t)}-1\right|^2\\
&\leq |\log(|\phi_{X_1}(t)|)|+\frac{3}{4}\\
&\leq 2\Lambda+\frac{\Sigma^2}{h^2}+\frac{3}{4}
\end{align*}
holds. Here in the second line we used an elementary inequality
$|\log(1+z)-z|\leq |z|^2$ valid for $|z|<1/2,$ the third line follows from \eqref{logineqhalf},
while in the last line we used the bound
\begin{equation*}
|\log|\phi_X(t)||\leq
2\Lambda+{\Sigma^2}/{h^2}
\end{equation*}
which holds for $t\in[-\sqrt{2}h^{-1},\sqrt{2}h^{-1}].$
The result is now immediate from Conditions \ref{conditionh} and \ref{conditionm}, because on the set $B_n^c$ an upper bound on $|\log(|\hat{\phi}(t)|)|$ grows slower than $M_n.$
\end{proof}

\begin{proof}[Proof of Theorem \ref{thmsigmatilde2}]
A general line of the proof is similar to that of Theorem 2.1 in \cite{gug}, although the details and actual computations are different. We have
\begin{equation*}
\ex[(\hat{\sigma}_n^2-\sigma^2)^2]=\ex[(\hat{\sigma}_n^2-\sigma^2)^2 1_{B_n}]+\ex[(\hat{\sigma}_n^2-\sigma^2)^2 1_{B_n^c}]
=S_1+S_2,
\end{equation*}
where the two sets $B_n$ and $B_n^c$ are defined in \eqref{bnset} and $\delta$ in their definition is given by $\delta=(1/4)e^{-2\Lambda-\Sigma^2/h^2}.$ The term $S_1$ in the above display can be bounded as follows,
\begin{align*}
S_1&\lesssim \left(M_n^2\left(\int_{\mathbb{R}}|v_h(t)|dt\right)^2+\Sigma^4\right)\operatorname{P}(B_n)\\
&\lesssim \left(M_n^2\left(\int_{\mathbb{R}}|v_h(t)|dt\right)^2+\Sigma^4\right)\frac{e^{2\Sigma^2/h^2}}{nh^2}\\
&=\left(M_n^2h^4\left(\int_{\mathbb{R}}|v(t)|dt\right)^2+\Sigma^4\right)\frac{e^{2\Sigma^2/h^2}}{nh^2}\\
&\lesssim m_n^2 \frac{e^{2\Sigma^2/h^2}}{nh^2},
\end{align*}
where we used Chebyshev's inequality and Theorem \ref{thm-ineq} with $r=2$ to see the second line. Next we consider $S_2.$ By Lemma \ref{lemmaa1} on the set $B_n^c$ for all large enough $n$ truncation in the definition of
$\hat{\sigma}_n^2$ becomes unimportant and we have
\begin{align*}
S_2&=\ex\left[\left(\int_{\mathbb{R}}\log(|\hat{\phi}(t)|)v_h(t)dt-\sigma^2\right)^21_{B_n^c}\right]\\
&=\ex\left[\left(\int_{\mathbb{R}}\log\left(\left|\frac{\hat{\phi}(t)}{\phi_{X_1}(t)}\right|\right)v_h(t)dt+\int_{\mathbb{R}}\log(|\phi_{X_1}(t)|)v_h(t)dt-\sigma^2\right)^21_{B_n^c}\right].
\end{align*}
Hence by equation (4) in \cite{gug3}, the $c_2$-inequality and Conditions \ref{conditionf2} and \ref{conditionv2} we obtain that
\begin{align*}
S_2 & \lesssim \Lambda^2\left(\int_{\mathbb{R}}\Re(\phi_f(t))v_h(t)dt\right)^2\\
& + \ex\left[\left(\int_{\mathbb{R}}\log\left(\left|\frac{\hat{\phi}(t)}{\phi_X(t)}\right|\right)v_h(t)dt\right)^21_{B_n^c}\right]\\
& = S_3+S_4.
\end{align*}
To bound $S_3,$ we proceed as follows,
\begin{align*}
S_3 & \lesssim h^6 \int_{\mathbb{R}} |\phi_f(t)|^2e^{2\alpha |t|^s}dt \int_{\mathbb{R}\backslash [-h^{-1},h^{-1}]} e^{-2\alpha |t|^{s}}dt\\
&\lesssim h^6 \int_{1/h}^{\infty} e^{-2\alpha t^s}dt\\
&\lesssim h^{s+5} e^{-2\alpha/h^s},
\end{align*}
where we used the Cauchy-Schwarz inequality, the fact that $|\Re(\phi_f(t))|\leq |\phi_f(t)|$
and Condition \ref{conditionf2}. As far as $S_4$ is concerned, we have
\begin{align*}
S_4 & \lesssim \ex\left[\left(\int_{\mathbb{R}}\left|\frac{\hat{\phi}(t)}{\phi_{X_1}(t)}-1\right||v_h(t)|dt\right)^2 1_{B_n^c}\right]\\
& + \ex\left[ \left(\int_{\mathbb{R}}\left\{\log\left(\left|\frac{\hat{\phi}(t)}{\phi_{X_1}(t)}\right|\right)-\left(\left|\frac{\hat{\phi}(t)}{\phi_{X_1}(t)}\right|-1\right)\right\}v_h(t)dt\right)^2 1_{B_n^c} \right]\\
& = S_5+S_6.
\end{align*}
An application of the Cauchy-Schwarz inequality and Conditions \ref{conditionsigma} and \ref{conditionf2}  give
\begin{equation}
\label{starstar}
S_5\lesssim
e^{4\Lambda+2\Sigma^2/h^2}\int_{\mathbb{R}}(v_h(t))^2dt\ex\left[\int_{-\sqrt{2}/h}^{\sqrt{2}/h}|\hat{\phi}(t)-\phi_{X_1}(t)|^2dt\right],
\end{equation}
where we also used the fact that on the set $B_n^c$ the inequality \eqref{logineqhalf} holds. Parseval's identity and Proposition 1.7 of \cite{tsyb} (notice that in the latter it is actually not necessary to have a positive kernel) applied to
the sinc kernel then yield
\begin{equation*}
\ex\left[\int_{-\sqrt{2}/h}^{\sqrt{2}/h}|\hat{\phi}(t)-\phi_{X_1}(t)|^2dt\right]\lesssim
\frac{1}{nh},
\end{equation*}
from which and from \eqref{starstar} we obtain
\begin{equation*}
S_5\lesssim e^{2\Sigma^2/h^2}h^4\frac{1}{n}.
\end{equation*}
Using the fact that on the set $B_n^c$ the inequality \eqref{logineqhalf} holds and combining it with an inequality $|\log(1+z)-z|\leq|z|^2$ valid for $|z|<1/2,$ one sees that $S_6\lesssim S_5.$ Furthermore, by a standard argument under Condition \ref{conditionh2} the term $S_3$ dominates other terms. For instance, we have
\begin{equation*}
e^{2\Sigma^2/h^2}h^4\frac{1}{n} h^{-s-5}e^{2\alpha/h^s}\rightarrow 0,
\end{equation*}
because
\begin{equation*}
\frac{2\Sigma^2}{h^2}+\frac{2\alpha}{h^s}-\log n+\ \log h^{-s-1}
=-(\log\log n)^2-(s+1)\log h\rightarrow -\infty.
\end{equation*}
This follows from \eqref{bandwidth2} and the fact that under Condition \ref{conditionh2} it holds that $h\asymp (\log n)^{-1/2}.$ The latter can be shown as formula (27) in \cite{but}. Hence $S_3$ dominates $S_5.$ Combination of all the above bounds on the terms $S_i$ completes the proof.
\end{proof}

\begin{proof}[Proof of Theorem \ref{thm-rho2}]
A general line of the proof is the same as in the proof of Theorem \ref{thm-rho}. With the same notation for the individual terms $T_i$ as in the latter, by the the same argument as for the term $T_1$ in the proof of Theorem \ref{thm-rho} and term $S_3$ in the proof of Theorem \ref{thmsigmatilde2} we have
\begin{align*}
T_1 & \lesssim \left( \int_{\mathbb{R}\backslash [-h^{-1},h^{-1}]} |\phi_f(t)|dt \right)^2 + h^{s-1}e^{-2\alpha/h^s} \\
& =  \left( \int_{\mathbb{R}\backslash [-h^{-1},h^{-1}]} e^{-\alpha |t|^s}e^{\alpha|t|^s}|\phi_f(t)|dt \right)^2 + h^{s-1}e^{-2\alpha/h^s}\\
& \lesssim \int_{\mathbb{R}\backslash [-h^{-1},h^{-1}]} e^{-2\alpha |t|^s} dt\int_{\mathbb{R}\backslash [-h^{-1},h^{-1}]} e^{2\alpha|t|^s}|\phi_f(t)|^2dt + h^{s-1}e^{-2\alpha/h^s}\\
& \lesssim \int_{\mathbb{R}\backslash [-h^{-1},h^{-1}]} e^{-2\alpha |t|^s} dt + h^{s-1}e^{-2\alpha/h^s}\\
& \lesssim \int_{1/h}^{\infty} e^{-2\alpha t^s}dt + h^{s-1}e^{-2\alpha/h^s}\\
& \lesssim h^{s-1}e^{-2\alpha/h^s}.
\end{align*}
Denote by $\operatorname{MSE}[\hat{\sigma}^2]$ the mean square error of $\hat{\sigma}^2.$ From the proof of Theorem \ref{thm-rho} and by Theorem \ref{thmsigmatilde2} we have
\begin{align*}
T_2 & \lesssim \frac{1}{h^2}\operatorname{MSE}[\hat{\sigma}^2] + T_4\\
& \lesssim h^{s-1}e^{-2\alpha/h^s} +  \frac{e^{2\Sigma^2/h^2}}{nh^8}.
\end{align*}
We then obtain
\begin{equation*}
T_1+T_2\lesssim h^{s-1}e^{-2\alpha/h^s}+\frac{e^{2\Sigma^2/h^2}}{nh^8}\lesssim h^{s-1}e^{-2\alpha/h^s},
\end{equation*}
because
\begin{equation*}
\frac{e^{2\Sigma^2/h^2}}{nh^8}h^{1-s}e^{2\alpha/h^s}\rightarrow 0,
\end{equation*}
which can be seen by taking the logarithm of the lefthand side and then using Condition \ref{conditionf2} and the fact that $h\asymp (\log n)^{-1/2},$ cf.\ formula (27) in \cite{but}, to conclude that the lefthand side in the above display diverges to minus infinity. This entails the first statement of the theorem.

Before proving the second statement of the theorem, we will show that the choice of $h$ as in Condition \ref{conditionh2} is optimal in a sense that it asymptotically minimises the order bound on the mean square error of $\hat{\rho}.$ This follows in essence by arguments similar to those used in the proof of Lemma 4 in \cite{but}: a minimiser $h_{\ast}$ with respect to $h$ of the expression
\begin{equation*}
h^{s-1}e^{-2\alpha/h^s}+\frac{e^{2\Sigma^2/h^2}}{nh^8},
\end{equation*}
which up to a constant is an upper bound on the risk of the estimator $\hat{\rho},$ can be found from the equation
\begin{equation*}
\frac{d}{dh}\left[h^{s-1}e^{-2\alpha/h^s}+\frac{e^{2\Sigma^2/h^2}}{nh^8}\right]=0.
\end{equation*}
After neglecting lower order terms  (here we assume that $h\rightarrow 0$ as $n\rightarrow\infty$), one can deduce that $h_{\ast}$ has to satisfy
\begin{equation}
\label{opband}
\frac{2\alpha s}{4\Sigma^2}nh_{\ast}^{9}(1+o(1))=e^{2\Sigma^2/h_{\ast}^2+2\alpha/h_{\ast}^s}.
\end{equation}
Taking logarithm of the both sides of \eqref{opband} yields that $h_{\ast}$ satisfies
\begin{equation}
\label{opband2}
a\log h_{\ast} + \frac{2\alpha}{h_{\ast}^s}+\frac{2\Sigma^2}{h_{\ast}^2}=\log n + C(1+o(1))
\end{equation}
for some constants $a$ and $C,$ cf.\ equation (11) in \cite{but}. With $h_{\ast}$ chosen as in \eqref{opband} or \eqref{opband2}, the term $h_{\ast}^{s-1}e^{-2\alpha/h_{\ast}^s}$ dominates the term $e^{2\Sigma^2/h^2}/(nh_{\ast}^8),$ cf.\ pp.\ 30--31 in \cite{but} for a similar result for the kernel-type deconvolution density estimator in a particular deconvolution problem. Indeed, for $h_{\ast}$ satisfying \eqref{opband} we have
\begin{equation*}
h_{\ast}^{s-1}e^{-2\alpha/h_{\ast}^s} \asymp \frac{e^{2\Sigma^2/h_{\ast}^2}}{nh_{\ast}^8}h_{\ast}^{s-2},
\end{equation*}
and it suffices to observe that $s<2$ by our assumptions. Let $\tilde{h}$ be as in \eqref{bandwidth2}. For any $b\in\mathbb{R}$ the formula
\begin{equation*}
h_{\ast}^be^{-2\alpha/h_{\ast}^s}=\tilde{h}^be^{-2\alpha/\tilde{h}^s}(1+o(1))
\end{equation*}
holds, which can be proved exactly as formula (28) of \cite{but}. Furthermore,
\begin{equation*}
\frac{e^{2\Sigma^2/\tilde{h}^2}}{n\tilde{h}^8}=o\left( e^{-2\alpha/\tilde{h}^s} \right),
\end{equation*}
which is a direct consequence of \eqref{bandwidth2} and the fact that $\tilde{h}\asymp (\log n)^{-1/2},$ cf.\ formula (27) of \cite{but}. Finally,
\begin{equation*}
\frac{e^{2\Sigma^2/\tilde{h}^2}}{n\tilde{h}^8}\leq \frac{e^{2\Sigma^2/h_{\ast}^2}}{nh_{\ast}^8},
\end{equation*}
for $n$ large enough, which can be shown as formula (30) of \cite{but}. These facts together imply that $h$ as in \eqref{bandwidth2} defines an optimal bandwidth, for an upper bound on the risk of $\hat{\rho}(x)$ computed with such an $h$ is of the same order as the one computed with $h_{\ast}.$ Combination of the above results proves the first statement of the theorem.

To complete the proof of the theorem, it remains to prove \eqref{s1}. Assuming $n$ is large enough, by \eqref{bandwidth2} it holds in the case $s=1$ that
\begin{equation*}
\frac{1}{h}=\frac{-\alpha+\sqrt{2\Sigma^2(\log n-(\log \log n)^2)+\alpha^2}}{2\Sigma^2}.
\end{equation*}
From this is follows that
\begin{equation*}
\exp\left(-\frac{2\alpha}{h}\right)\lesssim \exp\left( -{2\alpha}\sqrt{\frac{1}{2\Sigma^2}(\log n-(\log\log n)^2)} \right).
\end{equation*}
The righthand side is of order $\exp(-2\alpha\sqrt{(\log n)/(2\Sigma^2)}),$ as can be seen by some straightforward manipulations: we have
\begin{equation}
\label{eq2}
\begin{split}
\exp\left( -{2\alpha}\sqrt{\frac{1}{2\Sigma^2}(\log n-(\log\log n)^2)} \right)
= \exp\left( -{2\alpha}\sqrt{\frac{1}{2\Sigma^2}\log n} \right)\\
\times \exp \left( {2\alpha}\sqrt{\frac{1}{2\Sigma^2}\log n} -{2\alpha}\sqrt{\frac{1}{2\Sigma^2}(\log n-(\log\log n)^2)}  \right)
\end{split}
\end{equation}
and
\begin{equation*}
\sqrt{\frac{1}{2\Sigma^2}\log n} -\sqrt{\frac{1}{2\Sigma^2}(\log n-(\log\log n)^2)}\rightarrow 0,
\end{equation*}
because the lefthand side of the latter can be rewritten as
\begin{equation*}
\sqrt{\frac{1}{2\Sigma^2}}\frac{(\log\log n)^2}{\sqrt{\log n}} \left[  \left( 1 - \sqrt{ 1- \frac{(\log\log n)^2}{\log n}  } \right)\frac{\log n}{(\log\log n)^2}  \right].
\end{equation*}
The term in the square brackets converges to $-1/2,$ because it converges to a derivative of the function $\sqrt{1-t}$ at $t=0,$ while for the first factor we have
\begin{equation*}
\sqrt{\frac{1}{\sqrt{2\Sigma^2}}}\frac{(\log\log n)^2}{\sqrt{\log n}}\rightarrow 0.
\end{equation*}
Consequently, the righthand side of \eqref{eq2} is of order $\exp(-2\alpha\sqrt{(\log n)/(2\Sigma^2)}).$ This concludes the proof of the theorem.
\end{proof}



\section*{Acknowledgments}
\noindent{Part of the research reported in this paper was done while the author was at EURANDOM, Eindhoven, The Netherlands.}


\begin{thebibliography}{34}

\bibitem{ait-sahalia} Y.\ A\"it-Sahalia and J.\ Jacod.  Volatility estimators for discretely sampled L\'evy processes. \emph{Ann. Statist.}, 35:355--392, 2007.

\bibitem{akritas1} M.G. Akritas and R.A. Johnson. Asymptotic inference in {L}\'evy processes of the discontinuous type. \emph{Ann. Statist.}, 9:604--614, 1981.

\bibitem{akritas2} M.G. Akritas. Asymptotic theory for estimating the parameters of a {L}\'evy process. \emph{Ann. Inst. Statist. Math.}, 34:259--280, 1982.

\bibitem{basawa1} I.V. Basawa and P.J. Brockwell. Inference for gamma and stable processes. \emph{Biometrika}, 65:129--133, 1978.

\bibitem{basawa2} I.V. Basawa and P.J. Brockwell. A note on estimation for gamma and stable processes. \emph{Biometrika}, 67:234--236, 1980.

\bibitem{belomestny} D.\ Belomestny and M.\ Rei\ss. {Spectral calibration of exponential {L}\'evy
models}. \emph{Finance Stoch.}, 10:449-474, 2006.

\bibitem{belom} D.\ Belomestny and M.\ Rei\ss. {Spectral calibration of exponential {L}\'evy
models [2]}. SFB 649 Discussion Paper 2006-035, 2006.

\bibitem{bertoin} J.\ Bertoin. \emph{L\'evy Processes}.\ Cambridge University Press, Cambridge, 1996.

\bibitem{bibby} B.M. Bibby and M.\ S{\o}rensen. A hyperbolic diffusion model for stock prices. \emph{Finance Stoch.}, 1:25--41, 1997.

\bibitem{blaesild} P.~Bl{\ae}sild and M.~S{\o}rensen. {HYP} --- a computer program for analyzing data by means of the
  hyperbolic distribution. Research report No.\ 248, Department of Mathematical Statistics, Aarhus University, 1992.

\bibitem{haerdle} S.~Borak, W.~H\"{a}rdle and R.~Weron. Stable distributions. In P.~Cizek, W.~H\"ardle and R.~Weron (editors), \emph{Statistical
  Tools for Finance and Insurance}, 21--44, Springer, Berlin, 2005.

\bibitem{brown} L.D.\ Brown, M.G.\ Low and L.H.\ Zhao. Superefficiency in nonparametric function estimation. \emph{Ann.\ Statist.}, 25:2607-2625, 1997.

\bibitem{buchmann} B.\ {Buchmann}.\ Weighted empirical processes in the nonparametric inference for {L}\'evy processes.\ \emph{Math.\ Methods Statist.}, 18:281--309, 2009.

\bibitem{bu}
B.\ Buchmann and R.\ {Gr\"{u}bel}.\ Decompounding: an estimation problem for {P}oisson random sums.\ \emph{Ann.\ Statist.}, 31:1054--1074, 2003.

\bibitem{bugr} B.\ Buchmann and R.\ Gr\"{u}bel.\ Decompounding {P}oisson random sums: recursively truncated estimates
  in the discrete case. \emph{Ann.\ Inst.\ Statist.\ Math.}, 56:743--756, 2004.

\bibitem{burnaev} E.V.\ Burnaev.\ Inversion formula for infinitely divisible distributions. \emph{Russ.\ Math.\ Surv.}, 61:772--774, 2006.

\bibitem{matias} C.\ Butucea and C.\ Matias. Minimax estimation of the noise level and of the deconvolution density in a semiparametric convolution model. \emph{Bernoulli}, 11:309--340, 2005.

\bibitem{but} C.\ Butucea and A.B.\ Tsybakov. Sharp optimality for density deconvolution with dominating bias, I. \emph{Theory Probab.\ Appl.}, 52:24--39, 2008.

\bibitem{tsyb1} C.\ Butucea and A.B.\ Tsybakov. Sharp optimality for density deconvolution with dominating bias, II. \emph{Theory Probab.\ Appl.}, 52:237--249, 2008.

\bibitem{carr} P.\ Carr, H.\ Geman, D.B.\ Madan, and M.\ Yor.\ The fine structure of asset returns: an empirical investigation.\ \emph{J.\ Bus.}, 75:305--332, 2002.

\bibitem{chen} S.X.\ Chen, A.\ Delaigle and P.\ Hall. Nonparametric estimation for a class of L\'evy processes. {\em J.\ Econometrics}, 157:257--271, 2010.

\bibitem{chung} K.L. Chung. \emph{A Course in Probability Theory}, 3rd edition.\ Academic Press, New York, 2001.

\bibitem{comtecatalot} F.\ Comte and V.\ Genon-Catalot.\ Nonparametric estimation for pure jump L\'evy processes based on high frequency data.\  \emph{Stochastic Process.\ Appl.}, 119:4088-4123, 2009.

\bibitem{comte} F.\ Comte and V.\ Genon-Catalot.\ Nonparametric adaptive estimation for pure jump L\'evy processes.\ \emph{Ann.\ Inst.\ H.\ Poincar\'e Probab.\ Statist.}, 46:595--617, 2010.

\bibitem{catalot} F.\ Comte and V.\ Genon-Catalot.\ Non-parametric estimation for pure jump irregularly sampled or noisy L\'evy processes. \emph{Stat.\ Neerl.}, 64: 290--313, 2010.

\bibitem{genon} F.\ Comte and V.\ Genon-Catalot. Estimation for L\'evy processes from high frequency data within a long time interval. \emph{Ann.\ Statist.}, 39: 803--837, 2011.

\bibitem{lacour} F.\ Comte and C.\ Lacour. Data driven density estimation in presence of unknown convolution operator. To appear in \emph{J. R. Stat. Soc. Ser. B Stat. Methodol.}, 2011.

\bibitem{tankov} R.\ Cont and P.\ Tankov. \emph{Financial Modelling with
Jump Processes}. Chapman \& Hall/CRC, Boca Raton, 2003.

\bibitem{cont} R.\ Cont and P.\ Tankov. Retrieving L\'evy processes from option prices: regularization of an ill-posed inverse problem. \emph{SIAM J.\ Control Optim.}, 45:1--25, 2006.

\bibitem{delaigle0} A.\ Delaigle. An alternative view of the deconvolution problem.\ \emph{Statist.\
Sinica}, 18:1025--1045, 2008.

\bibitem{devroye} L.\ Devroye. On the non-consistency of an estimate of Chiu.\ \emph{Stat.\ Probab.\ Lett.}, 20:183--188, 1994.

\bibitem{gyorfi} L.\ Devroye and L.\ Gy\"{o}rfi. \emph{Nonparametric Density Estimation: the $L_1$ View}. John Wiley \& Sons, New York, 1985.

\bibitem{shota} B.\ van Es and S. Gugushvili. Asymptotic normality of the deconvolution kernel density estimator under the vanishing error variance. \emph{J.\ Korean Statist. Soc.}, 39: 102--115, 2010.

\bibitem{gug}
B.\ van Es, S.\ Gugushvili and P.\ Spreij.\ A kernel type nonparametric density estimator for decompounding. \emph{Bernoulli}, 13:672--694, 2007.

\bibitem{fan1} J.~Fan. {On the optimal rates of convergence for nonparametric deconvolution problems}. \emph{Ann.\ Statist.}, 19:1257--1272, 1991.

\bibitem{fan2} J.~Fan. {Deconvolution with supersmooth distributions}. \emph{Canad.\ J.\ Statist.}, 20:155-169, 1992.

\bibitem{figueroalop} E. Figueroa-L\'opez. Sieve-based confidence intervals and bands for L\'evy densities. \emph{Bernoulli}, 17:643--670, 2011.

\bibitem{gug3} S.\ Gugushvili.\ Nonparametric estimation of the characteristic triplet of a discretely observed L\'evy process.\ \emph{J.\ Nonparametr.\ Stat.}, 21:321--343, 2009.

\bibitem{atomic} S. Gugushvili, B.\ van Es and P. Spreij. Deconvolution for an atomic distribution: rates of convergence. To appear in \emph{J. Nonparametr. Stat.}, 2011.

\bibitem{proceedings} S. Gugushvili, C. Klaassen and P. Spreij (editors). \emph{Statistical Inference for L\'evy Processes with Applications to Finance.} \emph{Stat. Neerl.}, 64(3), 2010. 

\bibitem{jongbloed} G.~Jongbloed and F.H. van~der Meulen. Parametric estimation for subordinators and induced {OU} processes. \emph{Scand. J. Statist.}, 33:825--847, 2006.

\bibitem{meulen} G.\ Jongbloed, F.H.\ van der Meulen and A.W.\ van der Vaart. Nonparametric inference for L\'evy-driven Ornstein-Uhlenbeck  processes.  \emph{Bernoulli}, 11:759-791, 2005.

\bibitem{kappus} J.\ Kappus and M.\ Rei\ss. Estimation of the characteristics of a L\'evy process observed at arbitrary frequency.\ \emph{Stat.\ Neerl.}, 64: 314--328, 2010.

\bibitem{kyprianou} A.E.\ Kyprianou. \emph{Introductory Lectures on Fluctuations of L\'evy Processes with Applications}. Springer, Berlin, 2006.

\bibitem{meister} A.\ Meister. Density estimation with normal measurement error
with unknown variance. \emph{Statist. Sinica}, 16:195--211, 2006.

\bibitem{merton} R.C. Merton. Option pricing when underlying stock returns
are discontinuous. \emph{J.\ Financ.\ Econ.}, 3:125--144, 1976.

\bibitem{neumann} M.H.\ Neumann.\ On the effect of estimating the error density in
nonparametric deconvolution.\ \emph{J.\ Nonparametr.\ Statist.}, 7:307--330, 1997.

\bibitem{reiss}
M.H.\ Neumann and M.\ Rei\ss.\ Nonparametric estimation for {L}\'evy processes from low-frequency
  observations.\ \emph{Bernoulli}, 15:223--248, 2009.

\bibitem{nolan} J.P. Nolan. Maximum likelihood estimation and diagnostics for stable
  distributions. In O.E. Barndorff-Nielsen, T.~Mikosch, and S.I. Resnick (editors), \emph{L\'evy Processes: Theory and Applications}, 379--400,
  Birkh\"auser, Boston, 2001.

\bibitem{ren} Y.-F.\ Ren and H.-Y.\ Liang. On the best constant in Marcinkiewicz-Zygmund inequality. \emph{Statist. Probab. Lett.}, 53:227-233, 1999.

\bibitem{rydberg} T.H. Rydberg. The normal inverse {G}aussian {L}\'evy process: simulation and approximation. \emph{Stoch. Models}, 13:887--910, 1997.

\bibitem{sato}
K.-I.\ Sato.\ \emph{L\'evy Processes and Infinitely Divisible Distributions}.\ Cambridge University Press, Cambridge, 2004.

\bibitem{sohl} J.\ S\"{o}hl.\ Polar sets for anisotropic Gaussian random fields.\ \emph{Statist.\ Probab.\ Lett.}, 80:840--847, 2010.

\bibitem{tsyb}
A.B.\ Tsybakov.\ {\em Introduction to Nonparametric Estimation}.\ Springer, New York, 2009.

\bibitem{vdvaart} A.W. van der Vaart. \emph{Asymptotic
Statistics}.
Cambridge University Press, Cambridge, 1998.

\bibitem{vaart} A.W.\ van der Vaart and J.A.\ Wellner. \emph{Weak Convergence and Empirical Processes with Applications to Statistics}. Springer, New York, 1996.

\bibitem{wand} M.P.\ Wand. Finite sample performance of deconvolving density estimators.\ \emph{Statist.\ Probab.\ Lett.}, 37:131--139, 1998.

\bibitem{watteel}
R.N.\ Watteel and R.J.\ Kulperger.\ Nonparametric estimation of the canonical measure for infinitely
  divisible distributions.\ \emph{J.\ Stat.\ Comput.\ Simul.}, 73:525--542, 2003.

\bibitem{zolotarev} V.M. Zolotarev. \emph{One-Dimensional Stable Distributions}. American Mathematical Society, Providence, 1986.

\end{thebibliography}
\end{document}